\documentclass[10pt,a4paper]{article}
\usepackage[T1]{fontenc}
\usepackage[english]{babel}
\usepackage{amsmath}
\usepackage{amsfonts}
\usepackage{amssymb}
\usepackage{amsthm}
\usepackage{titlesec}
\usepackage{indentfirst}
\usepackage{color}
\usepackage{enumerate}

\newtheorem{lemme}{Lemma}
\newtheorem*{lem}{Lemma}
\newtheorem*{lem1}{Menchoff's Lemma}
\newtheorem*{prop}{Proposition}
\newtheorem*{theo}{Theorem}
\newtheorem{theoreme}{Theorem}
\newtheorem{cor}{Corollary}
\newtheorem{proposition}{Proposition}

\theoremstyle{definition}
\newtheorem*{defi}{Definition}
\newtheorem{defin}{Definition}
\newtheorem*{rem}{Remark}
\newtheorem*{rems}{Remarks}
\newtheorem*{ex}{Example}

\newtheorem*{notas}{Notations}
\title{Some Banach spaces of Dirichlet series}
\date{}
\newcommand\dis{\displaystyle}
\newcommand\eps{\varepsilon}
\newcommand\ind{{\rm 1\kern-.30em I}}
\newcommand\e{{\rm e}}

\newcommand{\biindice}[3]%
{%

\begin{array}[t]{c}
{\displaystyle #1}\\
{\scriptstyle #2}\\
{\scriptstyle #3}
\end{array}

}

\titleformat{\section}[block]
{\normalfont\large\bfseries\filcenter}
{\thesection}
{1em}
{}

\titleformat{\subsection}[block]
{\normalfont\bfseries\filcenter}
{\thesubsection}
{1em}
{}

\renewcommand{\thesubsection}{\Alph{subsection}}

\begin{document}
\begin{center}
{\Large \bfseries Some Banach spaces of Dirichlet series }
\vspace{5pt}

{ \textit{ Maxime Bailleul, Pascal Lef\`evre}

 				   \par\vspace*{20pt}} 
\end{center}

\begin{abstract} The Hardy spaces of Dirichlet series denoted by $ \mathcal{H}^p$ ($p\geq 1$) have been studied in \cite{hedenmalm1995hilbert} when $p=2$ and in \cite{bayart2002hardy} for the general case. In this paper we study some $L^p$-generalizations of spaces of Dirichlet series, particularly two families of Bergman spaces denoted $ \mathcal{A}^p$ and $ \mathcal{B}^p$. We recover classical properties of spaces of analytic functions: boundedness of point evaluation, embeddings between these spaces and "Littlewood-Paley" formulas when $p=2$. We also show that the $ \mathcal{B}^p$ spaces have properties similar to the classical Bergman spaces of the unit disk while the $ \mathcal{A}^p$ spaces have a different behavior. 
\end{abstract}

\section{Introduction}
\subsection{Background and notations}

In \cite{hedenmalm1995hilbert}, the authors defined the Hardy space $ \mathcal{H}^2$ of Dirichlet series with square-summable coefficients. Thanks to the Cauchy-Schwarz inequality, it is a space of analytic functions on $ \mathbb{C}_{ \frac{1}{2}}:= \lbrace s \in \mathbb{C}, \, \Re(s) > \frac{1}{2} \rbrace$ and this domain is maximal. This space is isometrically isomorphic to the Hardy space $H^2( \mathbb{T}^{ \infty})$ (see \cite{cole1986representing} for the definition of $H^2(\mathbb{T}^{ \infty})$). 

F. Bayart introduced in \cite{bayart2002hardy} the more general class of Hardy spaces of Dirichlet series $ \mathcal{H}^p$ ($1 \leq p < + \infty$). We shall recall the definitions below.

In another direction, McCarthy defined in \cite{mccarthy2004hilbert} some weighted Hilbert spaces of two types: Bergman-like spaces and Dirichlet-like spaces.

It is the starting point of many recent researches on spaces of Dirichlet series, for instance in \cite{olsen2011local}, \cite{olsen2012boundary} and \cite{olsen2008local}, some local properties of these spaces are studied and in \cite{bayart2002hardy}, \cite{bayart2003compact}, \cite{lefevre2009essential}, \cite{queffelec2013approximation} and \cite{queffelec2013decay} some results about composition operators on these spaces are obtained. \\

We recall some known facts about Dirichlet series. The study of Dirichlet series may appear more complicated than the study of power series. For instance, there is a first important difference: all the notions of radius of convergence coincide for Taylor series but Dirichlet series has several abscissas of convergence. The two most standard ones are the abscissa of simple convergence $ \sigma_c$ and the abscissa of absolute convergence $\sigma_a$ (see \cite{queffelecBook}, \cite{tenenbaum1995introductiona}).

Let $f$ be a Dirichlet series of the following form 
\[ f(s) = \sum_{n=1}^{ + \infty } a_n n^{ -s} \quad (1). \]
We shall need the two other following abscissas:
\[ \left. \begin{array}{ccl}
\sigma_u(f) & = &\inf \lbrace a\;|\;\hbox{ The series }\,(1) \hbox{ is uniformly convergent for } \Re(s)>a \rbrace   \\
            & = & \hbox{abscissa of uniform convergence of } f.   \\
            &   &                                                 \\   
\sigma_b(f) & = &\inf \lbrace a\;|\; \hbox{ The function } f \hbox{ has an analytic, bounded extension for } \Re(s)>a \rbrace   \\
            & = & \hbox{abscissa of boundedness of } f.  
\end{array}\right. \]
Actually, the two previous abscissas coincide: for all Dirichlet series $f$, one has $\sigma_b(f)= \sigma_u(f)$ (see \cite{bohr1913uber}). This result due to Bohr is really important for the study of $ \mathcal{H}^{ \infty}$, the algebra of bounded Dirichlet series on the right half-plane $\mathbb{C}_+$ (see \cite{maurizi2010some}). We shall denote by $ \Vert\cdot\Vert_{\infty}$ the norm on this space:
\[ \Vert f \Vert_{ \infty} := \sup_{ \Re(s)>0} \vert f(s) \vert. \]

Let us recall now the principle of Bohr's point of view on Dirichlet series: let ${n\geq 2}$ be an integer, it can be written (uniquely) as a product of prime numbers \linebreak $n=p_1^{ \alpha_1} \cdot \cdot \, \, p_k^{ \alpha_k}$ where $p_1=2
, \, p_2=3 $ etc \dots If $s$ is a complex number and if we consider $z=(p_1^{-s}, \, p_2^{-s}, \dots )$, then we have by $(1)$ 
\[ f(s) = \sum_{n=1}^{ + \infty } a_n (p_1^{-s})^{ \alpha_1} \cdot \cdot \, \, (p_k^{-s})^{ \alpha_k} = \sum_{n=1}^{ + \infty } a_n \, z_1^{ \alpha_1} \cdot \cdot \, \, z_k^{ \alpha_k}.  \]
So we can see a Dirichlet series as a Fourier series on the infinite-dimensional polytorus $ \mathbb{T}^{ \infty}$. We shall denote this Fourier series $D(f)$. This correspondence is not just formal. For instance, let $ \mathbb{P}$ be the set of prime numbers, Bohr proved the next result.

\begin{theo}[\cite{bohr1913uber}]
Let $f$ be a Dirichlet series of the form $(1)$. Then  
\[ \sum_{ p \in \mathbb{P}} \vert a_p \vert \leq \Vert f \Vert_{\infty}. \]
\end{theo}

The infinite-dimensional polytorus $ \mathbb{T }^{ \infty}$ can be identified with the group of complex-valued characters $ \chi$ on the positive integers which satisfy the following properties  
\[\left\lbrace
\begin{array}{ll}
\vert \chi(n) \vert=1 &\quad \forall n \geq 1,  \\
 \chi(nm) = \chi(n) \, \chi(m) &\quad \forall n, \, m \geq 1. \end{array}   \right. \]
To obtain this identification for $ \chi = ( \chi_1, \, \chi_2, \dots) \in \mathbb{T}^{ \infty}$, it suffices to define $\chi$ on the prime numbers by $ \chi(p_i)= \chi_i$ and use multiplicativity. We shall denote by $m$ the normalized Haar measure on $ \mathbb{T}^{ \infty}$.  \\

Now, let us recall how one can define the Hardy spaces of Dirichlet series $\mathcal{H}^p$. We already precised the case $p=\infty$, nevertheless, it is easy to see that the following description also applies to the case $p=\infty$. We fix $p\ge1$. The space $H^p( \mathbb{T}^{ \infty})$ is the closure of the set of analytic polynomials with respect to the norm of $L^p( \mathbb{T}^{ \infty}, \, m)$. Let $f$ be a Dirichlet polynomial, by the Bohr's point of view $D(f)$ is an analytic polynomial on $ \mathbb{T}^{ \infty}$. By definition $ \Vert f \Vert_{ \mathcal{H}^p} := \Vert D(f) \Vert_{ H^p( \mathbb{T}^{ \infty})}$. The space $ \mathcal{H}^p$ is defined by taking the closure of Dirichlet polynomials with respect to this norm. Consequently $ \mathcal{H}^p$ and $ H^p( \mathbb{T}^{ \infty})$ are isometrically isomorphic. When $p=2$, $ \mathcal{H}^2$ is just the space of Dirichlet series of the form $(1)$ which verify 
\[ \displaystyle{ \sum_{n=1}^{ + \infty} \vert a_n  \vert^2 < + \infty}. \]

 Let $\mathcal{D}$ be the space of functions which admit representation by a convergent Dirichet series on some half-plane. When a function $f$ belongs to $\mathcal{D}$ and $\sigma>0$, we can define the function $f_{\sigma}\in\mathcal{D}$, the translate of $f$ by $\sigma$, i.e. ${f_{ \sigma}(s):= f (\sigma+s)}$. We can then define a map from $\mathcal{D}$ to $\mathcal{D}$ by $T_{ \sigma}(f)=f_{ \sigma}$. 

For $\theta \in \mathbb{R}$, $\mathbb{C}_{ \theta}$ is the half-plane defined by $ \lbrace s \in \mathbb{C}, \, \Re(s)> \theta \rbrace$.

We shall denote by $\mathcal{P}$ the space of Dirichlet polynomials, that is to say the vector space spanned by the functions $\dis\e_n(z)=n^{-z}$, where $n\ge1$. At last, in the sequel, for  $p\ge1$, we write $p'$ its conjugate exponent: $\displaystyle\frac{1}{p}+\frac{1}{p'}=1$.



\subsection{Organization of the paper}
In the present paper, we introduce two classes of Bergman spaces of Dirichlet series. We give some properties of these spaces, precise the growth of the point evaluation of functions belonging to these spaces. At last, we compare them to the Hardy spaces of Dirichlet series: it appears some very different phenomena.

\begin{defin}
Let $p \geq 1$, $P$ be a Dirichlet polynomial and $ \mu$ be a probability measure on $(0, + \infty)$ such that $0 \in Supp( \mu)$. We define 
\[ \Vert P \Vert_{ \mathcal{A}_{ \mu}^p} = \bigg{(} \int_{0}^{+ \infty} \Vert P_{ \sigma} \Vert_{ \mathcal{H}^p}^p \, d \mu( \sigma) \bigg{)}^{1/p}. \]
$\mathcal{A}_{\mu}^p$ will be the completion of $\mathcal{P}$ with respect to this norm. 

\noindent When $ \mu( \sigma) = 2 e^{-2 \sigma} \, d \sigma$, we denote these spaces simply $ \mathcal{A}^p$. More generally, let us fix $ \alpha>-1$ and consider the probability measure $ \mu_{ \alpha}$, defined by 
$$\dis d \mu_{ \alpha} ( \sigma) = \frac{2^{\alpha +1}}{\Gamma(\alpha+1)} \sigma^{ \alpha} \exp(-2 \sigma) \, d \sigma.$$
The space $\mathcal{A}_{\mu_\alpha}^p$ will be denoted simply $\mathcal{A}_{\alpha}^p$ in this case.
\end{defin}

\begin{defin}
On the infinite dimensional polydisk $\mathbb{D}^{ \infty}$, we consider the measure $A=\lambda\otimes \lambda \otimes \cdots$ where $ \lambda$ is the normalized Lebesgue measure on $ \mathbb{D}$. For $p\ge1$, the space $B^p( \mathbb{D}^{ \infty})$ is the closure of the set of analytic polynomials with respect to the norm of $L^p( \mathbb{D}^{ \infty}, \, A)$. Let $f$ be a Dirichlet polynomial, we set $ \Vert f \Vert_{ \mathcal{B}^p} := \Vert D(f) \Vert_{ B^p( \mathbb{D}^{ \infty})}$. The space $ \mathcal{B}^p$ is defined by taking the closure of $ \mathcal{P}$ with respect to this norm.
\end{defin}

In section 2, we prove that the point evaluation is bounded on the spaces $ \mathcal{A}_{ \mu}^p$ for any $s \in \mathbb{C}_{1/2}$. More precisely, let $ \delta_s$ be the operator of point evaluation at $s \in  \mathbb{C}_{1/2}$, which is {\it a priori} defined for Dirichlet polynomials (or convergent Dirichlet series). We prove that the operator extends to a bounded operator which we still denote by $\delta_s$ and we obtain that there exists a constant $c_p$ such that for every $s \in \mathbb{C}_{1/2}$,
$$\Vert\delta_{s}\Vert_{({\mathcal{A}^p)}^*} \leq c_p\Big(\frac{\Re(s)}{2\Re(s)-1}\Big)^{2/p}.$$
We also show that the identity from $ \mathcal{H}^2$ to $ \mathcal{A}^p$ is not bounded when $p>2$ but is compact when $p=2$. Finally we obtain a Littlewood-Paley formula for the Hilbert spaces $ \mathcal{A}_{\mu}^2$. \\

In section 3, we prove that the point evaluation is bounded on the space $ \mathcal{B}^p$ for any $s \in \mathbb{C}_{1/2}$ and we have 
\[ \Vert \delta_s \Vert_{ { (\mathcal{B}^p)}^*} = \zeta(2 \Re(s))^{2/p}. \]
By a result of hypercontractivity, we obtain that the injection from $ \mathcal{H}^p$ to $ \mathcal{B}^{2p}$ is bounded. This phenomenon is similar to what happens in the classical framework of Hardy/Bergman spaces in one variable. Nevertheless, concerning compactness, we have the following curiosity: the injection from $ \mathcal{H}^p$ to $ \mathcal{B}^p$ is not compact. We also obtain a Littlewood-Paley formula for the space $ \mathcal{B}^2$.

\section{The Bergman spaces $ \mathcal{A}_{\mu}^p$}
\subsection{Hilbert spaces of Dirichlet series with weighted $\ell_2$ norm}
First, we recall some facts of \cite{mccarthy2004hilbert}. We changed the definition in order to include the constants in these spaces, which seems to us more natural. \\

Let $w=(w_n)_{ n \geq 1}$ be a sequence of positive numbers, the space $ \mathcal{A}_w^2$ is defined by 

\[ \mathcal{A}_w^2 := \bigg{\lbrace} f \in \mathcal{D}, \, f(s) = \sum_{n=1}^{ + \infty} a_n n^{-s}, \, \sum_{n=1}^{ + \infty} \vert a_n \vert^2 w_n < + \infty\ \bigg{ \rbrace}. \]

Of course, if $w \equiv 1$, $ \mathcal{A}_w^2$ is just the classical Hardy space $ \mathcal{H}^2$. In order to obtain good properties for these spaces, we need to impose some properties on the weights. 

\begin{defi}{\cite{mccarthy2004hilbert}}
Let $ \mu$ be a probability measure on $(0, + \infty)$ such that ${0 \in Supp( \mu)}$. We define for $n \geq 1$ 
\[w_n := \int_{0}^{ + \infty} n^{-2 \sigma} d \mu( \sigma). \]
In this case, we say that the space $ \mathcal{A}_{ \mu}^2:= \mathcal{A}_w^2$ is a (hilbertian) Bergman-like space and that $w$ is a Bergman weight.
\end{defi}

\begin{ex}
When $\mu=\delta_0$, the Dirac mass at point $0$, we get the Hardy space $\mathcal{H}^2$. In the opposite situation, when $\mu(\{0\})=0$, it is easy to see that the sequence $w$ converges to $0$. 

In the case $\mu=\mu_\alpha$, where $\alpha>-1$,  we have $w_n =  (\log(n)+1)^{-1- \alpha}$ for $n \geq 1$ and the associated space is $ \mathcal{A}_{ \alpha}^2$. For $ \alpha=0$, we recover the space $ \mathcal{A}^2$ and we can notice that the limit (degenerated) case $ \alpha=-1$ corresponds to $\mathcal{H}^2$.  
\end{ex}

Mac Carthy proved that these spaces are spaces of analytic functions on $ \mathbb{C}_{ \frac{1}{2}}$. It is a consequence of the following lemma:

\begin{lem}[\cite{mccarthy2004hilbert}]
Let $w$ be a Bergman weight. Then $w$ is non-increasing and $w$ decreases more slowly than any negative power of $n$, that is to say:
\[\forall \varepsilon >0, \exists c>0, \, w_n > c \, n^{ - \varepsilon} \quad \forall n \geq 1. \]
\end{lem}

In addition, $\mathbb{C}_{\frac{1}{2}}$ is a maximal domain. Indeed let us consider the Riemann Zeta function $ \zeta$ (\cite{tenenbaum1995introductiona}), for every $ \varepsilon>0$ and every weight $w$:
\[ \zeta\Big( \frac{1}{2}+s +  \varepsilon\Big) = \sum_{n=1}^{+ \infty} \frac{1}{n^{1/2+ \varepsilon+ s}} \in \mathcal{A}_w^2. \]
But these Dirichlet series admit a pole at $ \frac{1}{2} - \varepsilon$.
\medskip


\subsection{Point evaluation on $ \mathcal{A}_{\mu}^p$}

First we can easily compute the norm of the evaluation in the case of the Hilbert spaces $ \mathcal{A}_{\mu}^2$. In this case the point evaluation is bounded on $ \mathbb{C}_{1/2}$, it is optimal and the reproducing kernel at $s \in \mathbb{C}_{1/2}$ is 
\[ K_{ \mu}(s,w):= \sum_{n=1}^{+ \infty}  \frac{n^{-w- \overline{s}}}{w_n}\]
and $ \Vert \delta_s \Vert_{ {(\mathcal{A}_{ \mu}^2)}^*} \leq \bigg{(} \displaystyle{\sum_{n=1}^{+ \infty}  \frac{n^{-2 \sigma}}{w_n} }\bigg{)}^{1/2}$ for every $s= \sigma +it \in \mathbb{C}_{1/2}$.

In the general case, the next theorem provides us a majorization which gives the right order of growth when the abscissa is close to the critical value $1/2$. Actually we are going to distinguish the behavior according to the valuation of the function, so we shall need some estimates according the constant coefficient vanishes or not. It would be interesting to work with truncated functions with higher order {\it i.e.} when each $a_n=0$ for $n\le v$, nevertheless we shall only concentrate on the case $v=0$ and $v=1$, because these are the only needed cases in this paper.

\begin{defin}

Let $\dis\mathcal{H}_{\infty}^p$ be the subspace of $\dis\mathcal{H}^p$ of functions whose valuation is at least $1$, i.e. the space of Dirichlet series whose constant coefficient $a_1$  vanishes (remember that $a_1$ is actually the value at infinity, and this explains our notation).

Let $\dis\mathcal{A}_{ \mu,\infty}^p$ be the subspace of $\dis\mathcal{A}_{\mu}^p$ of functions whose constant coefficient vanishes. In the particular case of the measure $\dis  \mu_{ \alpha}$, we write  $\dis\mathcal{A}_{\alpha,\infty}^p$. At last, when $\alpha=0$, we simply use the natural notation $\dis\mathcal{A}_{\infty}^p$.
\end{defin}

On the spaces $\dis\mathcal{H}^p$ (resp. $\dis\mathcal{H}_{\infty}^p$), we define $\dis\Delta_p(s)$ (resp. $\dis\Delta_{p,\infty}(s)$)  as the norm of the evaluation at point $\dis s\in\mathbb{C}_{1/2}$. We recall that  we know from \cite{bayart2002hardy} that $\Delta_p(s) = \zeta( 2\Re(s))^{1/p}$.

\begin{theoreme}{\label{evalApmieux}}
Let $p \geq 1$ and $\mu$ be a probability measure on $(0, + \infty)$ such that $0 \in Supp( \mu)$. 

Then the point evaluation is bounded on ${\cal P}\cap\mathcal{A}_{ \mu}^p$ (resp. on ${\cal P}\cap\mathcal{A}_{ \mu,\infty}^p$) for any $s \in \mathbb{C}_{1/2}$. Hence it extends to a bounded operator on $\mathcal{A}_{ \mu}^p$ (resp. on $\mathcal{A}_{ \mu,\infty}^p$) whose norm verifies
 
 \begin{enumerate}[(i)]
\item $\dis\Vert \delta_{s} \Vert_{ {(\mathcal{A}_{ \mu}^p)}^*} \leq \inf_{\eta \in (0, \Re(s)-1/2)} \bigg{(} \frac{\Vert \Delta_p(\Re(s) - \, {\tiny\bullet})\Vert_{L^{p'}([0,\, \Re(s)-1/2- \eta], \, d\mu)} }{ \mu([0, \Re(s)-1/2 - \eta])} \bigg{)}.$

\item $\dis\Vert \delta_{s} \Vert_{ {(\mathcal{A}_{ \mu,\infty}^p)}^*} \leq \inf_{\eta \in (0, \Re(s)-1/2)} \bigg{(} \frac{\Vert \Delta_{p,\infty}(\Re(s) - \, {\tiny\bullet})\Vert_{L^{p'}([0,\, \Re(s)-1/2- \eta], \, d\mu)} }{ \mu([0, \Re(s)-1/2 - \eta])} \bigg{)}.$
\end{enumerate}
\end{theoreme}

\begin{proof}
We prove only $(i)$ since the proof for $(ii)$ is the same. 
Let us fix $\eta$ in $(0,\Re(s)-1/2)$. We can assume that $s=\sigma \in (1/2, + \infty)$ thanks to the vertical translation invariance of the norm on $ {\mathcal{A}_{ \mu}^p}$. Let $P$ be a Dirichlet polynomial. We have 
\[ P( \sigma) = P_{ \varepsilon}( \sigma - \varepsilon) \quad \hbox{ for any } \varepsilon \in (0, \sigma -1/2). \]
We know that the point evaluation is bounded on $ \mathcal{H}^p$:
\[ \vert P( \sigma) \vert \leq\Delta_p(\sigma - \varepsilon)  \Vert P_{ \varepsilon} \Vert_{ \mathcal{H}^p} \quad \hbox{ for any } \varepsilon \in (0, \sigma -1/2). \]
By integration on $( 0, \sigma -1/2 - \eta )$ we obtain 
\[ \mu( [0, \sigma -1/2- \eta]) \vert P( \sigma) \vert \leq \int_{ 0}^{ \sigma-1/2 - \eta}  \Delta_p(\sigma - \varepsilon)  \Vert P_{ \varepsilon} \Vert_{ \mathcal{H}^p} \, d\mu( \varepsilon)\,. \]

Then, by H\"older's inequality,
\[ \mu( [0, \sigma -1/2- \eta]) \vert P( \sigma) \vert \leq \Vert P \Vert_{ \mathcal{A}_{ \mu}^p}\,\cdot\,\|\Delta_p(\sigma - {\tiny\bullet}\, )\|_{L^{p'}([0,\, \sigma -1/2- \eta], d \mu)}\;.\]
Since  $\eta\in(0,\Re(s)-1/2)$ is arbitrary, the result follows.\end{proof}

\begin{cor}{\label{corevalmieux}}
Let $p \geq 1$ and $ \alpha>-1$. 
\begin{enumerate}[(i)]

\item The point evaluation is bounded on $ \mathcal{A}_{\alpha}^p$ for any $s \in \mathbb{C}_{ 1/2}$ and there exists a positive constant $c_{p,\alpha}$ such that for every $s \in \mathbb{C}_{1/2}$ we have:
\[ \Vert\delta_{s} \Vert_{{(\mathcal{A}_{\alpha}^p)}^*} \leq c_{p,\alpha} \bigg{(}\frac{\Re(s)}{2 \Re(s)-1)} \bigg{)}^{(2+\alpha)/p}\cdot\]
\smallskip

\item The point evaluation is bounded on $ \mathcal{A}_{\alpha,\infty}^p$ for any $s \in \mathbb{C}_{ 1/2}$ and there exists a positive constant $c'_{p,\alpha}$ such that for every $s \in \mathbb{C}_{1/2}$ we have:
\[ \Vert\delta_{s} \Vert_{{(\mathcal{A}_{\alpha,\infty}^p)}^*} \leq \frac{c'_{p,\alpha}}{(2 \Re(s)-1)^{(2+\alpha)/p}}\cdot\]
\end{enumerate}
\end{cor}

\begin{proof}
In this proof, we shall use that, for every $x>1$: $\displaystyle\zeta(x) \leq \frac{x}{x-1}\cdot$ On the other hand, in the sequel,  $A\lesssim B$ means that there exists some constant $c$ depending  on $p$ and $\alpha$ only such that $A\le cB$. 
\\

Fix $s = \sigma \in (1/2, + \infty)$ and $ \eta \in (0, \sigma - 1/2)$. In our framework, there exists some constant $C_\alpha$ depending on $\alpha$ only, such that, for every $A>0$:
\[ \mu_{\alpha}([0,A])\ge C_\alpha\min\big(1, A^{ \alpha + 1}\big). \]

Let us prove $(i)$.

We first consider the case $p=1$. We choose $\eta = ( \sigma-1/2)/2$. Since
$$\sup_{ \varepsilon  \in [ 0, \, (\sigma -1/2)/2]} |\zeta( 2 \sigma - 2 \, \varepsilon) | =  \zeta( \sigma + 1/2)\le\frac{2\sigma+1}{2\sigma-1}$$
the conclusion follows from the preceding theorem.

Now let us assume that $p>1$ and $p \neq 2$ (we already know exactly the norm of the evaluation in this case). We have
 
$$ \int_{0}^{ \sigma -1/2- \eta} \zeta(2 \sigma- 2 \varepsilon)^{p'/p} \, d \mu_{\alpha}( \varepsilon) \lesssim (2\sigma)^{p'/p}\int_{0}^{ \sigma-1/2- \eta} \frac{\varepsilon^{\alpha}}{(2 \sigma-2 \varepsilon-1)^{p'/p}} e^{-2 \varepsilon}\, d \varepsilon. $$

We split our discussion in two cases, according to $p>2$ or $2>p>1$.

First let us assume that $p>2$. We have $p'/p<1$ hence the previous integral converges for $\eta = 0$ and is majorized by
\[ \int_{0}^{ \sigma-1/2} \frac{\varepsilon^{\alpha}}{(2 \sigma-2 \varepsilon-1)^{p'/p}} \, d \varepsilon = \frac{1}{(2 \sigma-1)^{p'/p}} \cdot \int_{0}^{ \sigma-1/2} \frac{\varepsilon^{\alpha}}{\big{(}1-(2 \varepsilon)/(2 \sigma-1)\big{)}^{p'/p}} \, d \varepsilon \]
\[ = \frac{(2 \sigma-1)^{\alpha+1}}{2^{ \alpha+1}(2 \sigma-1)^{p'/p}} \cdot \int_{0}^{ 1} t^{\alpha}(1-t)^{-p'/p} \, d t \quad \hbox{ with } t= \frac{2 \varepsilon}{2 \sigma-1}\cdot \]
Finally we obtain:
\[  \int_{0}^{ \sigma -1/2- \eta} \zeta(2 \sigma - 2 \varepsilon)^{p'/p} \, d \mu_{\alpha}( \varepsilon) \lesssim (2\sigma)^{p'/p} \frac{B( \alpha+1, 1-p'/p)}{(2 \sigma-1)^{(p'/p)- \alpha-1}}  \]
where $B$ is the classical Beta function (\cite{davis1972gamma}). Finally with the choice $ \eta=0$ in Theorem \ref{evalApmieux}, we obtain 
$$\Vert \delta_{ \sigma} \Vert_{ {(\mathcal{A}_{\alpha}^p)}^*} \dis\lesssim\bigg{(} \frac{B( \alpha+1, 1-p'/p)}{(2 \sigma-1)^{(p'/p)- \alpha-1}} \bigg{)}^{1/p'} \cdot \frac{(2\sigma)^{1/p}}{\min\big(1,(\sigma-1/2)^{\alpha+1}\big)}.$$
This estimation is good when $\sigma$ is bounded (and more precisely when $\sigma$ is close to $1/2$). We have to look at the asymptotic behavior. So, coming back to the integral and considering $\sigma\ge1$, we have
$$ \int_{0}^{ \sigma -1/2} \zeta(2 \sigma- 2 \varepsilon)^{p'/p} \, d \mu_{\alpha}( \varepsilon)$$
we majorize it by
$$ \int_{0}^{ \sigma -1} \sup_{x\ge2}|\zeta(x)|^{p'/p} \, d \mu_{\alpha}(\varepsilon) +\int_{ \sigma -1}^{ \sigma -1/2} \zeta(2 \sigma- 2 \varepsilon)^{p'/p} \, d \mu_{\alpha}( \varepsilon).$$
The first integral is uniformly bounded relatively to $\sigma$ and the second one is majorized by 
$$\int_{ \sigma -1}^{ \sigma -1/2} \frac{\varepsilon^{\alpha}(2\sigma)^{p'/p}}{(2 \sigma-2 \varepsilon-1)^{p'/p}} e^{-2 \varepsilon}\, d \varepsilon\lesssim \sigma^{\alpha+p'/p}e^{-2 \sigma}\int_{ 0}^{1} \frac{1}{u^{p'/p}}\, du\lesssim1.  $$

It proves that the norm of the evaluation is uniformly bounded when $\sigma>1$. Gathering everything, the conclusion follows and $(i)$ is proved when $p>2$.
\smallskip

Now for the  case  $1<p<2$, we have $p'/p>1$ and we cannot choose $\eta =0$ because the integral is not convergent. But in fact, it suffices to choose the middle point $ \eta = ( \sigma-1/2)/2$. We conclude in the same way. \\ 
\smallskip

Let us prove $(ii)$. Obviously $\dis\Vert\delta_{s} \Vert_{{(\mathcal{A}_{\alpha,\infty}^p)}^*}\le\Vert\delta_{s} \Vert_{{(\mathcal{A}_{\alpha}^p)}^*}$, hence the conclusion follows from $(i)$ when the real part of $s$ is bounded by $1$. 

It suffices to look at the behavior when $\sigma>1$ and it will follow from the (asymptotic) behavior of $\Delta_{p,\infty}$:
$$\dis\Delta_{p,\infty}(s)\le\frac{1}{\Re(s)-1}.$$
Indeed, for every $f\in{\cal P}\cap\mathcal{H}_{\infty}^p\subset\mathcal{H}_{\infty}^1$, we have for any $s\in\mathbb{C}_1$:
$$f(s)=\lim_{T\rightarrow+\infty}\int_{-T}^T\overline{\widetilde{\zeta}(s+it)} f(it).$$
where $\dis\widetilde{\zeta}(z)=\sum_{n\ge2}n^{-z}.$
Hence
$$| f(s)|\le\|\widetilde{\zeta}_\sigma\|_{\mathcal{H}^\infty}\|f\|_{\mathcal{H}^1}\le\frac{1}{\sigma-1}\|f\|_{\mathcal{H}^p}.$$

Now, the sequel of the proof follows the lines of the proof of $(i)$ so we leave the details to the reader.\end{proof}

\begin{rems}
$\hbox{  }$

 \begin{enumerate}[(i)]
\item
Let us precise why it is optimal in many cases: the behavior of $\dis\|\delta_{s}\|_{{(\mathcal{A}_{\alpha}^p)}^*}$ around the critical line $\sigma=1/2$ cannot be a power of $\dis\frac{\Re(s)}{2\Re(s)-1}$ better than $(2+\alpha)/p$. Indeed, let $\sigma>1/2$, we would like to consider the function $\dis(\zeta_\sigma)^{2/p}$. Let us mention that we can define the function $\zeta^q$ (where $q>0$) through the Euler product:
$$\zeta^q(z)=\dis\prod_{p\in\mathbb P}\Big[\frac{1}{1-p^{-z}}\Big]^q.$$
Actually we first work with $F$ being a partial sum of $\dis(\zeta_\sigma)^{2/p}$,
we obtain:
$$|F(\sigma)|^p \le\|\delta_{\sigma}\|_{{(\mathcal{A}_{\alpha}^p)}^*}^p\|F\|_{\mathcal{A}_{\alpha}^p}^p \lesssim\|\delta_{\sigma}\|_{{(\mathcal{A}_{\alpha}^p)}^*}^p\int_0^{+\infty}\|(F_{\eps})\|^p_{\mathcal{H}^p} \eps^\alpha\exp(-2\eps)\,d\eps $$
because $F$ is a Dirichlet polynomial. Now if we assume that $p>1$, we know (see \cite{aleman2012fourier}) that $(\e_n)_{n \geq 1}$ is a Schauder basis for $\mathcal{H}^p$ hence there exists $c_p>0$ such that:
$$|F(\sigma)|^p \lesssim c_p \|\delta_{\sigma}\|_{{(\mathcal{A}_{\alpha}^p)}^*}^p\int_0^{+\infty}\|\zeta_{\sigma + \eps}^{2/p}\|^p_{\mathcal{H}^p} \eps^\alpha\exp(-2\eps)\,d\eps. $$

But 
$$\dis\|(\zeta_{\sigma+\eps})^{2/p}\|^p_{\mathcal{H}^p}=\|\zeta_{\sigma+\eps}\|^2_{\mathcal{H}^2}=\zeta(2\sigma+2\eps)$$
and we get (since $F$ was an arbitrary partial sum of $\zeta_{\sigma}^{2/p}$):
$$|\zeta(2\sigma)|^2\lesssim\|\delta_{\sigma}\|_{{(\mathcal{A}_{\alpha}^p)}^*}^p\sum_{n\ge1}\frac{n^{-2\sigma}}{(1+\ln(n))^{\alpha+1}}.$$

When $\alpha<0$, we get $\dis|\zeta(2\sigma)|^2\lesssim\|\delta_{\sigma}\|_{{(\mathcal{A}_{\alpha}^p)}^*}^p (2\sigma-1)^\alpha$
hence 
$$\dis\frac{1}{(2 \sigma-1)^{(2+\alpha)}}\lesssim \|\delta_{\sigma}\|_{{(\mathcal{A}_{\alpha}^p)}^*}^p$$
which proves our claim, in a strong way: the majorization in $(i)$ of Cor.\ref{corevalmieux} is actually also (up to a constant) a minoration.

When $\alpha\ge0$, we have $$\dis\frac{1}{(2 \sigma-1)^{(2+\alpha)}|\log(2 \sigma-1)|}\lesssim \|\delta_{\sigma}\|_{{(\mathcal{A}_{\alpha}^p)}^*}^p$$
which proves that we cannot get a better exponent than $(2+\alpha)/p$ in $(i)$, Cor.\ref{corevalmieux}.
\item
Let $\sigma>1/2$ and $\mu= \mu_{\alpha}$, we already know that the reproducing kernel at $\sigma$ is defined by
\[ K_{\mu_{\alpha}}(\sigma,w)= \sum_{n=1}^{+ \infty} (1+ \log(n))^{\alpha+1} n^{-\sigma-w} \quad, \, \forall w \in \mathbb{C}_{1/2}.\]
Then
\[ K_{\mu_{\alpha}}(\sigma,\sigma) \leq \Vert \delta_{\sigma} \Vert_{ (\mathcal{A}_{\alpha}^2)^{*}} \Vert K_{\mu_{\alpha}}(\sigma,\, \bullet \, ) \Vert_{ \mathcal{A}_{\alpha}^2} \]
and by the property of the reproducing kernel
\[ K_{\mu_{\alpha}}(\sigma,\sigma)^{1/2} \leq \Vert \delta_{\sigma} \Vert_{ (\mathcal{A}_{\alpha}^2)^{*}}.\]
The converse inequality is already known, then 
\[ \Vert \delta_{\sigma} \Vert_{ (\mathcal{A}_{\alpha}^2)^{*}} = K_{\mu_{\alpha}}(\sigma,\sigma)^{1/2} = \bigg{(} \frac{\Gamma(2+ \alpha)}{(2 \sigma-1)^{2+ \alpha}} + O(1) \bigg{)}^{1/2} \]
when $\sigma$ goes to $1/2$ (see \cite{olsen2008local} for the second equality) and so our result is sharp when $p=2$.

\item With the same notations, we have 
\[  K_{\mu_{\alpha}}(\sigma, \sigma)^2 \leq \Vert \delta_{ \sigma} \Vert_{ (\mathcal{A}_{\alpha}^1)^{*}} \Vert K_{\mu_{\alpha}}(\sigma, \, \bullet \, )^2 \Vert_{ \mathcal{A}_{\alpha}^1} = \Vert \delta_{ \sigma} \Vert_{ (\mathcal{A}_{\alpha}^1)^{*}} \Vert K_{\mu_{\alpha}}(\sigma, \, \bullet \, ) \Vert_{ \mathcal{A}_{\alpha}^2}^2 \]
and again by the property of the reproducing kernel, we obtain
\[ K_{\mu_{\alpha}}(\sigma, \sigma) \leq \Vert \delta_{ \sigma} \Vert_{ (\mathcal{A}_{\alpha}^1)^{*}}. \]
We conclude as in (ii) and so the result is also sharp for $p=1$.

\item In (i), we used that $(\e_n)_{n \geq 1}$ is a Schauder basis for $\mathcal{H}^p$ when $p>1$. This result is also true for $\mathcal{A}_{\mu}^p$ when $p>1$: just use the result on $\mathcal{H}^p$, then it suffices to make an integration and use the density of the Dirichlet polynomials. This remark is also true for the spaces $ \mathcal{B}^p$.

\end{enumerate}
\end{rems}
\color{black}

Let us mention here that we are able to give a more precise majorization in the particular case of an even integer  $p$: constants are equal to $1$. It immediately follows from a general method explained in annexe at the end of our paper:

\begin{proposition}{\label{propevalpair}}
Let $p$ be an even integer, $\mu$ be as in Th.\ref{evalApmieux}. 

\begin{enumerate}[(i)]

\item For every $s \in \mathbb{C}_{1/2}$ we have: 
$$\Vert\delta_{s} \Vert_{{(\mathcal{A}_{\mu}^p)}^*} \leq\Vert\delta_{s} \Vert_{{(\mathcal{A}_{\mu}^2)}^*}^{2/p}\;.$$

\noindent In particular, 

\item For every $s \in \mathbb{C}_{1/2}$ we have: 
$$\Vert\delta_{s} \Vert_{{(\mathcal{A}^p)}^*} \leq\Big((\zeta-\zeta')(2 \Re(s))\Big)^{1/p}\sim\frac{1}{(2 \Re(s)-1)^{2/p}}\quad\hbox{when }\Re(s)\rightarrow1/2.$$

\end{enumerate}

\end{proposition}

\medskip

As soon as a Bergman like space is defined, a Dirichlet like space is naturally associated:

\begin{defin}
Let $p \geq 1$ and $ \mu$ be a probability measure on $(0, + \infty)$. We define the Dirichlet space $ \mathcal{D}_{ \mu}^p$ as the space of Dirichlet series $f$ such that 
\[ \Vert f \Vert_{ \mathcal{D}_{ \mu}^p}^p := \vert f(+ \infty) \vert^p + \Vert f' \Vert_{ \mathcal{A}_{ \mu}^p}^p < + \infty. \]
Here $f( + \infty)$ stands for $\displaystyle{\lim_{ \Re(s) \rightarrow + \infty}} f(s) = a_1$, where $f$ has an expansion $(1)$.
\end{defin}

\begin{theoreme}{\label{evaldir}}
Let $p \geq 1$ and $ \mu$ be a probability measure on $(0, + \infty)$. For any $s \in \mathbb{C}_{1/2}$, we have 
\[ \vert f(s) \vert \leq 2^{1/p'} max \bigg{(} 1, \int_{\Re(s)}^{+ \infty} \Vert \delta_{t} \Vert_{ {(\mathcal{A}_{ \mu,\infty}^p})^{*}} d t \bigg{)} \times \Vert f \Vert_{ \mathcal{D}_{ \mu}^p}. \]
\end{theoreme}

\begin{proof}
Without loss of generality we may assume that $s= \sigma \in (1/2, + \infty)$. Now 
\[ \vert f( \sigma) - f( + \infty) \vert = \bigg{ \vert} \int_{ \sigma}^{ + \infty} f'(t) dt \bigg{ \vert} \leq \int_{\sigma}^{ + \infty} \vert f'(t) \vert dt \leq \int_{ \sigma}^{ + \infty} \Vert \delta_t \Vert_{{( \mathcal{A}_{ \mu,\infty}^p)}^{*}} dt \times \Vert f' \Vert_{ \mathcal{A}_{ \mu,\infty}^p}\]
since the constant coefficient of $f'$ vanishes, i.e. $f'\in \mathcal{A}_{ \mu,\infty}^p$. So we get 
\[ \vert f( \sigma) \vert \leq \vert f( + \infty) \vert + \int_{ \sigma}^{ + \infty} \Vert \delta_t \Vert_{{( \mathcal{A}_{ \mu,\infty}^p)}^{*}} dt \times \Vert f' \Vert_{ \mathcal{A}_{ \mu,\infty}^p} \]
\[ \leq \Big{(} 1 + \Big{(} \int_{ \sigma}^{ + \infty} \Vert \delta_t \Vert_{{( \mathcal{A}_{ \mu,\infty}^p)}^{*}} dt \Big{)}^{p'} \, \Big{)}^{1/p'} \times \Big{(} \vert f( + \infty) \vert^p + \Vert f' \Vert_{ \mathcal{A}_{ \mu,\infty}^p}^p \Big{)}^{1/p} \]
thanks to the H\"older's inequality. Now it suffices to remark that 
\[ \Big{(} 1 + \Big{(} \int_{ \sigma}^{ + \infty} \Vert \delta_t \Vert_{{( \mathcal{A}_{ \mu,\infty}^p)}^{*}} dt \Big{)}^{p'} \, \Big{)}^{1/p'} \leq 2^{1/p'} max \bigg{(} 1, \int_{\Re(s)}^{+ \infty} \Vert \delta_{ t} \Vert_{ {(\mathcal{A}_{ \mu,\infty}^p)}^{*}} d t \bigg{)}. \]
\end{proof}

\begin{cor}{\label{evaldiralpha}}
Let $ \alpha > -1$ and $p \geq 1$. There exists $c_{p,\alpha}>0$ such that for every $s \in \mathbb{C}_{1/2}$, we have 
\[\Vert \delta_s \Vert_{ \mathcal{D}_{ \alpha}^p} \leq \left\lbrace
\begin{array}{cl}
c_{p,\alpha}\dis\frac{1}{( 2 \Re(s)-1 )^{ \frac{(2+ \alpha)}{p} -1 }}&\qquad \hbox{ if } \alpha \neq p-2.  \\
                                                                            \\
c_{p,\alpha} \,\log(2 \Re(s)-1) &\qquad\hbox{ if } \alpha = p-2. \end{array}   \right. \]
\end{cor}

\begin{proof}
The proof follows from Th. \ref{evaldir} and corollary \ref{corevalmieux}.
\end{proof}

Let us make a digression. The proofs of theorems \ref{evalApmieux} and \ref{evaldir} are based on the fact that we work with Bergman spaces with axial weights. Replacing axial weights by radial weights, the same idea can be adapted to classical Bergman and Dirichlet spaces on the unit disk $ \mathbb{D}$. Let us precise here how we can easily estimate the norm of the evaluation on weighted spaces of analytic functions over the unit disc.

Let $ \sigma : (0,1) \rightarrow (0, + \infty)$ be a continuous function such that $ \sigma \in L^1(0,1)$. We extend it on $ \mathbb{D}$ by $\sigma(z)= \sigma(\vert z \vert)$. For $p \geq 1$, we consider the weighted Bergman space
\[ A_{ \sigma}^p := H( \mathbb{D}) \cap L^p( \mathbb{D}, \sigma( \vert z \vert) d \lambda(z)) \]
where $H( \mathbb{D})$ is the set of analytic functions on $ \mathbb{D}$ and $\lambda$ is the normalized Lebesgue measure on $ \mathbb{D}$. This space is equipped with the norm 
\[ \Vert f \Vert_{ A_{ \sigma}^p} = \bigg{(} \int_{ \mathbb{D}} \vert f(z) \vert^p \, \sigma(z) d \lambda(z) \bigg{)}^{1/p}. \]
We also consider the Dirichlet space $ D_{ \sigma}^p( \mathbb{D})$: it is the space of analytic functions on $ \mathbb{D}$ such that the derivative belongs to $A_{ \sigma}^p$. This space is equipped with the norm 
\[ \Vert f \Vert_{D_{ \sigma}^p} = \bigg{(} \vert f(0) \vert^p + \Vert f' \Vert_{ A_{ \sigma}^p}^p \bigg{)}^{1/p}. \]

We know that the point evaluation at $z \in \mathbb{D}$ is bounded on the spaces $H^p$ (see \cite{cowen1995composition}) and we have 
\[ \Vert \delta_z \Vert_{H^p} = \frac{1}{(1- \vert z \vert^2)^{1/p}}. \]

\begin{theoreme}
Let $p \geq 1$ and $z \in \mathbb{D}$. The point evaluation at $z$ is bounded on $A_{\sigma}^p$ and we have 
\[ \Vert \delta_{z} \Vert_{{(A_{\sigma}^p)}^{*}} \leq \inf_{ \eta \in (0, 1-|z|)} \bigg{(} \frac{\Vert r \rightarrow (1-(|z|/r)^2)^{-1/p} \Vert_{L^{p'}([|z|+ \eta, \, 1], \, \sigma(r) dr)}}{ S([|z|+ \eta, \, 1])} \bigg{)} \]  

where $\dis S(I)=\int_I\sigma(r)dr$.
\end{theoreme}

\begin{ex}
When $\sigma \equiv 1$ (The classical Bergman space $A^p$), we recover
\[ \Vert \delta_z \Vert_{ {(A^p)}^{*}} \lesssim \frac{1}{(1- \vert z \vert^2)^{2/p}} \quad \hbox{ for any } z \in \mathbb{D}. \] 
\end{ex}

\begin{theoreme}
Let $p \geq 1$ and $z \in \mathbb{D}$. We have 
\[ \Vert \delta_z \Vert_{{(D_{\sigma}^p)}^{*}} \leq 2^{1/p'} max \bigg{(} 1, \int_{0}^{\vert z \vert} \Vert \delta_{ r} \Vert_{ (A_{\sigma}^p)^{*}} \, d r \bigg{)}.
\]
\end{theoreme}
\subsection{$\mathcal{A}_{ \mu}^p$ is a space of Dirichlet series}

The results of the preceding section allow us to define, for each $s\in \mathbb{C}_{1/2}$, the value of $f\in\mathcal{A}_{ \mu}^p$ at $s$ as $\delta_s(f)$. Of course, it coincides with the natural definition when $f$ is a Dirichlet polynomial or when $f\in{\cal D}\cap \mathcal{A}_{ \mu}^p$. Now we want more: we wish to check that we are actually working on spaces of Dirichlet series. 

We first need the following tool.

\begin{lemme}{\label{Tepsi}}
Let $ \varepsilon > 0$ and $\mu$ be a probability measure on $(0, + \infty)$. Then 
$$T_{\varepsilon}\quad\left|\;\begin{matrix}
 {\cal P}\cap\mathcal{A}_{ \mu}^1 &\longrightarrow&  \mathcal{A}_{ \mu}^2  \cr
f&\longmapsto    &  f_\eps \cr 
\end{matrix}\right.$$
 is bounded.
 
This extends to a bounded operator (still denoted  $T_{\varepsilon}$) from $\mathcal{A}_{ \mu}^1$ to $\mathcal{A}_{ \mu}^2$.
\end{lemme}

In the proof, we shall use the following sequence.
\begin{defi}
Let $\mu$ be a probability measure on $(0, + \infty)$ with $0 \in Supp(\mu)$. We define 
\[ \widetilde{w}_n := \int_{0}^{ + \infty} n^{- \sigma} d \mu( \sigma) \quad \forall n \geq 1. \]
\end{defi}

\begin{proof}

We shall introduce three bounded operators.

First we define $S_{1} : {\cal P}\cap\mathcal{A}_{ \mu}^1 \rightarrow \mathcal{H}^1$ by 
\[ S_1 \bigg{(} \sum_{n = 1}^{+ \infty} a_n \e_n \bigg{)} := \sum_{n=1}^{ + \infty} a_n \widetilde{w}_n \e_n. \]
$S_1$ is bounded because for any Dirichlet polynomial we have 
\[ \bigg{\Vert} \sum_{n=1}^{N} a_n \widetilde{w}_n \e_n \bigg{\Vert}_{ \mathcal{H}^1} = \int_{ \mathbb{T}^{ \infty}} \bigg{\vert} \sum_{n=1}^{ N} a_n \widetilde{w}_n z_1^{ \alpha_1} \cdot \cdot \, \,  z_k^{ \alpha_k} \bigg{ \vert} dm(z)  \]
\[ =  \int_{ \mathbb{T}^{ \infty}} \bigg{\vert} \int_{0}^{ + \infty} \sum_{n=1}^{ N} a_n n^{- \sigma} z_1^{ \alpha_1} \cdot \cdot \, \,  z_k^{ \alpha_k} \, d \mu( \sigma) \bigg{ \vert} dm(z) \quad \hbox{ by definition of } \widetilde{w}_n.  \]
\[ \leq \int_{0}^{ + \infty} \int_{ \mathbb{T}^{ \infty}}  \bigg{\vert} \sum_{n=1}^{N}  a_n n^{- \sigma} z_1^{ \alpha_1} \cdot \cdot \, \,  z_k^{ \alpha_k}  \bigg{ \vert} \, dm(z) d \mu( \sigma) = \bigg{\Vert} \sum_{n=1}^{N} a_n \e_n \bigg{\Vert}_{ \mathcal{A}_{ \mu}^1}. \]
By density, this operator extends to a bounded operator (still denoted $S_{1}$).

Now we define $S_{2} : \mathcal{H}^1 \rightarrow \mathcal{H}^2$ by 
\[ S_2  \bigg{(} \sum_{n = 1}^{+ \infty} a_n \e_n \bigg{)} := \sum_{n=1}^{ + \infty} \frac{a_n n^{- \varepsilon}}{\sqrt{\widetilde{w}_n}} \e_n. \]
$S_{2}$ is bounded because $T_{\varepsilon/2}:\mathcal{H}^1\rightarrow \mathcal{H}^2$ is bounded (see \cite{bayart2002hardy}) and because, there exists $C >0$ such that $ \widetilde{w}_n > C n^{- \varepsilon}$.\\

The third operator $S_{3}: \mathcal{H}^2\rightarrow \mathcal{A}_{ \mu}^2$ is defined by 
\[  S_3  \bigg{(} \sum_{n = 1}^{+ \infty} a_n \e_n \bigg{)} := \sum_{n = 1}^{+ \infty} \frac{a_n}{\sqrt{\widetilde{w}_n}} \e_n. \]
$S_3$ is bounded because $w_n \leq \widetilde{w}_n$ for all $n \geq 1$. 

\noindent Hence $S_3 \circ S_2 \circ S_1$ is bounded and clearly coincides with $T_{ \varepsilon}$.\end{proof}

\begin{theoreme}{\label{espseries}}
The space $\mathcal{A}_{\mu}^p$ is a space of Dirichlet series: every $f\in\mathcal{A}_{\mu}^p\in{\cal D}$ with $\sigma_u(f)\le1/2$.
\end{theoreme}

\begin{proof}
It is obvious when $p\ge2$ since then $ \mathcal{A}_{\mu}^p\subset\mathcal{A}_{\mu}^2$. When $1\le p<2$, since $\mathcal{A}_{\mu}^p\subset\mathcal{A}_{\mu}^1$, we only have to prove the conclusion of the theorem in the case $p=1$, but this follows from the preceding lemma. Indeed, let us fix $f\in\mathcal{A}_{\mu}^1$, $\alpha>1/2$ and $\eps=\alpha-1/2>0$. The function $T_\eps(f)$ belongs to $\mathcal{A}_{\mu}^2$ so we can write for every $z\in\mathbb{C}_{1/2}$
$$T_\eps(f)(z)=\dis\sum_{n\ge1}a_n^{(\eps)}\,n^{-z}.$$

On the other hand, $f$ is a limit of a sequence of Dirichlet polynomials $(P_k)_{k\in\mathbb N}$ relatively to the space $\mathcal{A}_{\mu}^1$. The continuity of $T_\eps$ implies that 
$T_\eps(f)$  is the limit of $\big(P_k(\eps+ \bullet)\big)_{k\in\mathbb N}$ relatively to the norm of $\mathcal{A}_{\mu}^2$. Invoking the continuity of the point evaluation both at $z+\eps$ and at $z$, we get
$$\dis f(z+\eps)=\lim P_k(\eps+z)=\lim T_\eps\big( P_k\big)(z)=T_\eps(f)(z)=\sum_{n\ge1}a_n^{(\eps)}\,n^{-z}.$$  
In particular, for every $s\in\mathbb{C}_{\alpha}$, we have (with $z=s-\eps\in\mathbb{C}_{1/2}$): $$f(s)=\dis\sum_{n\ge1}\big(a_n^{(\eps)}n^\eps\big)\,n^{-s}.$$

Actually the coefficient do not depend on $\eps$ (by uniqueness of the Dirichlet expansion). Since $\alpha>1/2$ is arbitrary, we get the conclusion.\end{proof}

An immediate corollary of this section is the following proposition
\begin{proposition}\label{TepsPA}
In view of the results of this section, we actually have for every $ \varepsilon > 0$ that 
$$T_{\varepsilon}\quad\left|\;\begin{matrix}
\mathcal{A}_{ \mu}^1 &\longrightarrow&  \mathcal{A}_{ \mu}^2  \cr
f&\longmapsto    &  f_\eps \cr 
\end{matrix}\right.$$
 is well defined and bounded.
 
Recall that $f_\eps(s)=f(s+\eps)=\delta_{s+\eps}(f)$.
\end{proposition}
\smallskip

It seems clear that $\mathcal{H}^p\subset\mathcal{A}^p_\mu$ for any $p\ge1$ and any $\mu$. Indeed, the following theorem precise this fact and that the way we may compute the norm remains valid for general functions of $\mathcal{A}^p_\mu$.

\begin{theoreme}\label{Thnormhardberg}
Let $p\ge1$ and $\mu$ a probability measure whose support contains $0$.

\begin{enumerate}[(i)]

\item $\mathcal{H}^p\subset\mathcal{A}^p_\mu$ and, for every $f\in\mathcal{H}^p$, we have $\dis\|f\|_{\mathcal{A}^p_\mu}\le\|f \|_{\mathcal{H}^p}$.

\item For every $f\in\mathcal{H}^p$, we have  $\dis\|f\|_{\mathcal{A}^p_\mu}=\Big(\int_0^{+\infty}\|f_\sigma\|_{\mathcal{H}^p}^p\,d\mu\Big)^{1/p}$.

\item For every $f\in\mathcal{A}^p_\mu$, we have  $\dis\|f\|_{\mathcal{A}^p_\mu}=\lim_{c\rightarrow0^+}\|f_c\|_{\mathcal{A}^p_\mu}$.
\end{enumerate}
\end{theoreme}

\begin{proof} For every Dirichlet polynomials $f$, we have $\dis\|f\|_{\mathcal{A}^p_\mu}\le\|f \|_{\mathcal{H}^p}$, since $\mu$ is a probability measure and ${\dis\|f \|_{\mathcal{H}^p}=\sup_{c>0}\|f_c\|_{\mathcal{H}^p}}$. Now, (in the spirit of the proof of Th.\ref{espseries}) a density argument, combined with the boundedness of point evaluation, allows to conclude easily the first assertion.

Let $f \in \mathcal{H}^p$ and $\varepsilon>0$. There exists a Dirichlet polynomial $P$ such that $\Vert f - P \Vert_{ \mathcal{H}^{p}} < \varepsilon$. By the first assertion $\Vert f - P \Vert_{ \mathcal{A}_{\mu}^{p}} < \varepsilon$ and then
\[ \Vert f \Vert_{ \mathcal{A}_{\mu}^p} \leq \varepsilon + \Vert P \Vert_{ \mathcal{A}_{\mu}^p} = \varepsilon + \Big{(} \int_{0}^{+ \infty} \Vert P_{\sigma} \Vert_{ \mathcal{H}^p}^p d \mu( \sigma) \Big{)}^{1/p} \]
\[ \leq \varepsilon + \Big{(} \int_{0}^{+ \infty} \Vert P_{\sigma} - f_{\sigma} \Vert_{ \mathcal{H}^p}^p d \mu( \sigma) \Big{)}^{1/p} + \Big{(} \int_{0}^{+ \infty} \Vert f_{\sigma} \Vert_{ \mathcal{H}^p}^p d \mu( \sigma) \Big{)}^{1/p}. \]
Now, $T_{ \sigma}$ is a contraction on $\mathcal{H}^p$  for every $\sigma>0$ and so 
\[ \Vert f \Vert_{ \mathcal{A}_{\mu}^p} \leq 2 \varepsilon + \Big{(} \int_{0}^{+ \infty} \Vert f_{\sigma} \Vert_{ \mathcal{H}^p}^p d \mu( \sigma) \Big{)}^{1/p}. \]
By the same way we obtain a lower bound and finally the second assertion.

For the third assertion, we shall use that $T_{c}$ is a contraction on $ \mathcal{A}_{\mu}^p$ for every $c>0$: as in the first assertion, it suffices to check it on Dirichlet polynomials but in this case this is clear by definition of the norm for Dirichlet polynomials and the fact that $T_{c}$ is a contraction on $\mathcal{H}^p$. Now let $f \in \mathcal{A}_{\mu}^p$ and $\varepsilon,c>0$. There exists $P$ a Dirichlet Polynomial such that $ \Vert f - P \Vert_{ \mathcal{A}_{\mu}^p}< \varepsilon$, then
\[ \Vert f - f_c \Vert_{ \mathcal{A}_{\mu}^p} \leq \Vert f- P \Vert_{ \mathcal{A}_{\mu}^p} + \Vert P- P_c \Vert_{ \mathcal{A}_{\mu}^p} + \Vert P_c-f_c \Vert_{ \mathcal{A}_{\mu}^p} \]
\[ \leq 2 \Vert f- P \Vert_{ \mathcal{A}_{\mu}^p} + \Vert P- P_c \Vert_{ \mathcal{A}_{\mu}^p} \leq 2 \varepsilon +  \Vert P- P_c \Vert_{ \mathcal{A}_{\mu}^p}\]
and by the dominated convergence theorem $\Vert P- P_c \Vert_{ \mathcal{A}_{\mu}^p}$ goes to $0$ when $c$ goes to $0$ and so the result is proved.
\end{proof}

\subsection{Vertical limits and Littlewood-Paley formula}
Let $f$ be a Dirichlet series absolutely convergent in a half-plane. For any sequences $( \tau_n) \subset \mathbb{R}$, we can consider vertical translations of $f$,
\[ (f_{ \tau_n}(s)):=(f(s + i \tau_n)).  \]
By Montel's theorem, this sequence is a normal family in the half-plane of absolute convergence of $f$ and so there exists a convergent subsequence  $\tilde{f}$. We say that  $\tilde{f}$ is a vertical limit of $f$. We shall use the next result.

\begin{prop}[\cite{hedenmalm1995hilbert}]
Let $f$ be a Dirichlet series of the form (1), absolutely convergent in a half-plane. The vertical limit functions of $f$ are exactly the functions of the form 
\[ f_{ \chi}(s) := \sum_{n=1}^{+ \infty} a_n \chi(n) \, n^{-s} \hbox{ where } \chi \in \mathbb{T}^{ \infty}. \]
\end{prop}

In \cite{hedenmalm1995hilbert} it is shown that every element $f$ in $ \mathcal{H}^2$ admits vertical limit functions $f_{ \chi}$ which converges $m$-almost everywhere on $ \mathbb{C}_+$. We have the same result with the Bergman spaces $ \mathcal{A}_{\mu}^p$. This is a consequence of the following lemma.

\begin{lem1}[\cite{olevskiaei1975fourier}]
Let $( \Omega, \mathcal{A}, \nu)$ be a probability space and $( \Phi_n)$ be an orthonormal sequence in $L^2( \Omega)$. Then 
\[ \sum_{n=1}^{ + \infty} \vert c_n \vert^2 \log^2(n) < + \infty \Rightarrow \sum_{n=1}^{+ \infty} c_n \Phi_n \hbox{ converge } \nu-ae. \]
\end{lem1}

\begin{proposition}{\label{limvert}}
Let $p \geq 1$, $\mu$ be a probability measure on $(0, + \infty)$ and let $f \in \mathcal{A}_{\mu}^p$ or $ \mathcal{D}_{ \mu}^p$. For almost all $\chi \in \mathbb{T}^{ \infty}$, $f_{ \chi}$ converges on $\mathbb{C}_+$.
\end{proposition}

\begin{rem}
It suffices to give the proof in the case of the Bergman-like spaces: indeed if $f$ is in $ \mathcal{D}_{ \mu}^p$ then $f' \in \mathcal{A}_{ \mu}^p$ and so $f_{\chi}^{'}$ converges on $ \mathbb{C}_+$ for almost every $\chi \in \mathbb{T}^{\infty}$ and the same hold for $f$.
\end{rem}

\begin{proof}
First, we prove the result when $p=2$. Let $f \in \mathcal{A}_{\mu}^2$ of the form $(1)$ and $c_n:= a_n n^{ - \sigma-it}$ where $ \sigma>0$ and $t \in \mathbb{R}$. Clearly $(\chi(n))$ is an orthonormal family in $L^2( \mathbb{T}^{ \infty})$. We have:
\[\sum_{n=2}^{ + \infty} \vert c_n \vert^2 \log^2(n) = \sum_{n=2}^{ +  \infty} \vert a_n \vert^2 w_n \bigg{(} \frac{\log^2(n)}{n^{2 \sigma} w_n} \bigg{)}.\]

If $w$ is a Bergman weight, we know that there exists a positive constant $C$ such that $w_n > C n^{- \sigma}$ for all $n \geq 1$, so 

\[ \frac{\log^2(n)}{n^{2 \sigma}w_n } \leq \frac{\log^2(n)}{C n^{  \sigma}}. \]

In this case, the right term of the previous inequality is finite and by the Menchoff's lemma, the proof is finished for $p=2$. 

Now we want to prove this result when $p \neq 2$. By inclusion on these spaces, it suffices to prove the result for $p=1$. 

Let $f \in \mathcal{A}_{ \mu}^1$. By  proposition \ref{TepsPA}, $f_{ \varepsilon} \in \mathcal{A}_{\mu}^2$ for every $ \varepsilon> 0$. So 
\[ \hbox{for every } \varepsilon> 0,  \hbox{ for  almost  all }  \chi \in \mathbb{T}^{ \infty}, \, (f_{ \varepsilon})_{ \chi} \hbox{ converges on } \mathbb{C}_+. \]
Then we have:
\[  \hbox{for every } n \geq 1, \hbox{ for  almost  all } \chi \in \mathbb{T}^{ \infty}, \, (f_{1/n})_{ \chi} \hbox{ converges on } \mathbb{C}_+. \]
Now we can invert the quantifiers:
\[\hbox{ for  almost  all } \chi \in \mathbb{T}^{ \infty}, \, \forall n \geq 1, \, (f_{1/n})_{ \chi} \hbox{ converges on } \mathbb{C}_+. \]
Of course if $(f_{1/n})_{\chi}$ converges on $ \mathbb{C}_+$ for every $n \geq 1$, $f_{ \chi}$ converges on $ \mathbb{C}_+$ and so we obtain result. \end{proof}

Now, following some ideas from \cite{kellay2012compact} in the case of the unit disk, we consider the case of the weighted Bergman-like spaces when $d \mu( \sigma)=h( \sigma) d \sigma$ where $h \geq 0$, $\Vert h \Vert_{L^1( \mathbb{R}^{+})}=1$ and $0 \in Supp(f)$. Let $w_h$ be the associated Bergman weight defined for $n \geq 1$ by 
\[ w_h(n) = \int_{0}^{ + \infty} n^{-2 \sigma} h( \sigma) d \sigma. \]

For $ \sigma>0$, we define 
\[ \beta_h ( \sigma) := \int_{0}^{ \sigma} ( \sigma-u) h(u) \, du  = \int_{0}^{ \sigma}  \int_{0}^{t} h(u) \, du dt. \]

\begin{rem}
Point out that $\dis\lim_{ \sigma \rightarrow + \infty} \beta_h( \sigma)n^{-2 \sigma} = 0$  for every $n \geq 2$. 
\end{rem}

We can compute the two first derivatives of $ \beta_h$:
\[ \left\lbrace
\begin{array}{lcc}
\beta_h'( \sigma) = \int_{0}^{ \sigma} h(u)du, \\
\beta_h''( \sigma) = h( \sigma).
\end{array}  \right.
 \]

In order to obtain a Littlewood-Paley formula for the spaces $ \mathcal{A}_{\mu}^2$, we need the following lemma.

\begin{lem}[\cite{bayart2003compact}]
Let $ \eta$ be a Borel probability measure on $ \mathbb{R}$ and $f$ of the form $(1)$. Then 
\[ \Vert f \Vert_{ \mathcal{H}^2}^2 = \int_{ \mathbb{T}^ { \infty}} \int_{ \mathbb{R}} \vert f_{ \chi}(it) \vert^2 d \eta(t) d m( \chi). \]
\end{lem}

\begin{theoreme}{\label{Little1}}["Littlewood-Paley formula"]
Let $ \eta$ be a Borel probability measure on $ \mathbb{R}$. Then 
\[ \Vert f \Vert_{ \mathcal{A}_{w_h}}^2 =  \vert f( + \infty) \vert^2 + 4 \int_{ \mathbb{T}^{ \infty}} \int_{ 0}^{ + \infty} \int_{ \mathbb{R}} \beta_h( \sigma) \vert f_{ \chi}^{'}( \sigma+it) \vert^2 \, d \eta(t) d \sigma d m ( \chi). \]
\end{theoreme}

\begin{proof}
Let $f \in \mathcal{A}_{w_h}^2$ of the form $(1)$, the previous lemma applied to $f_{ \sigma}$ where $ \sigma>0$ gives 
\[ \int_{ \mathbb{T}^{ \infty}} \int_{ \mathbb{R}} \vert f_{ \chi}^{'}( \sigma+it) \vert^2 d \eta(t) d m ( \chi) = \sum_{n=2}^{ + \infty} \vert a_n \vert^2 n^{-2 \sigma} \log^2(n) \quad \forall \sigma>0. \]
Now we multiply by $ \beta_h( \sigma)$ and integrate over $ \mathbb{R}_+$:
\[ \int_{ 0}^{ + \infty} \int_{ \mathbb{T}^ { \infty}} \int_{ \mathbb{R}} \beta_h( \sigma) \vert f_{ \chi}^{'}(\sigma + it) \vert^2 d \eta(t) d m( \chi)d \sigma = \sum_{n=2}^{ + \infty} \vert a_n \vert^2 \log^2(n) \, \int_{0}^{ + \infty} n^{-2 \sigma} \beta_h( \sigma) d \sigma. \]
Now, it suffices to prove that 
\[ w_h(n) =4  \log^2(n) \, \int_{0}^{ + \infty} n^{-2 \sigma} \beta_h( \sigma) d \sigma. \]
But by definition, we have
\[ w_h(n) = \int_{0}^{ + \infty} h( \sigma) n^{ -2 \sigma} d \sigma. \]
An integration by parts gives 
\[ w_h(n) = \bigg{ [} \int_{0}^{ \sigma} h(u) \, du \times n^{ - 2 \sigma} \bigg{]}_{0}^{ + \infty} + 2 \log(n) \int_{0}^{ + \infty}  \int_{0}^{ \sigma} h(u) \, du \times n^{- 2 \sigma} d \sigma \]
\[ = \bigg{ [} \beta_{h}^{'}( \sigma) \times n^{ - 2 \sigma} \bigg{]}_{0}^{ + \infty} + 2 \log(n) \int_{0}^{ + \infty} \beta_{h}^{'}( \sigma) \times n^{- 2 \sigma} d \sigma. \]
But we know that $ \beta_h^{'}(\sigma) \rightarrow 0$ when $\sigma \rightarrow 0$ (because $h \in L^1( \mathbb{R}_+)$) and $\beta_h^{'}(\sigma) \rightarrow \Vert h \Vert_1$ when $ \sigma \rightarrow + \infty$ . So we have 
\[  w_h(n) =  2 \log(n) \int_{0}^{ + \infty} \beta_{h}^{'}( \sigma) \times n^{- 2 \sigma} d \sigma. \]
Using again an integration by parts, we obtain:
 \[w_h(n) = 4\log^2(n) \, \int_{0}^{ + \infty} n^{-2 \sigma} \beta_h( \sigma) d \sigma.  \]
\end{proof}
 
\begin{ex}
let $ \alpha>-1$ and $f \in \mathcal{A}^2_{ \alpha}$, we have $\beta_h( \sigma) \approx \sigma^{ \alpha+2}$ when $ \sigma$ is small. Then 
\[ \Vert f \Vert_{ \alpha}^{2} \approx \vert f(+ \infty) \vert^2 + \frac{2^{ \alpha+3}}{\Gamma( \alpha+3)}  \int_{ \mathbb{T}^{ \infty}} \int_{ 0}^{ + \infty} \int_{ \mathbb{R}}  \sigma^{ \alpha+2} \vert f_{ \chi}^{'}( \sigma+it) \vert^2 \, d \eta(t) d  \sigma d m ( \chi). \]
\end{ex}

In the case of Dirichlet spaces, we obtain the following proposition.

\begin{proposition}
Let $d \mu= h d \sigma$ be a probability measure. Then for $f \in \mathcal{D}_{\mu}^2$ we have 
\[ \Vert f \Vert_{\mathcal{D}_{ \mu}^2}^2 = \vert f (+ \infty) \vert^2 + 4 \int_{ \mathbb{T}^{ \infty}} \int_{ 0}^{ + \infty} \int_{ \mathbb{R}} h( \sigma) \vert f_{ \chi}^{'}( \sigma+it) \vert^2 \, d \eta(t) d \sigma d m ( \chi). \]
\end{proposition}

\begin{rem}
These formulas are really useful to prove some criterion for compactness of composition operators (see \cite{bailleulcompo}). We can also use these formulas to compare $\mathcal{A}^2$ and $\mathcal{H}^2$ norms. For example, assume that $f \in \mathcal{A}_{\infty}^2$. Then for $x>0$, we have 
$$\Vert f \Vert_{ \mathcal{A}^2}^2 \ge 4 \int_{ 0}^{ x}  \int_{ \mathbb{T}^{ \infty}} \int_{ \mathbb{R}} \vert f_{ \chi}^{'}( \sigma+it) \vert^2 \, d \eta(t) dm( \chi) \, \sigma^2 d \sigma. $$
But we know that 
$$\dis\int_{ \mathbb{T}^{ \infty}} \int_{ \mathbb{R}} \vert f_{ \chi}^{'}( \sigma+it) \vert^2 \, d \eta(t) dm( \chi) = \Vert f_{ \sigma} \Vert_{ \mathcal{H}^2}^2 \geq \Vert f_{x} \Vert_{ \mathcal{H}^2}^2$$
when $ \sigma \leq x$. So we have 
\[ \Vert f \Vert_{ \mathcal{A}^2}^2 \geq 4 \Vert f_{x} \Vert_{ \mathcal{H}^2}^2 \times \int_{0}^{x} \sigma^2 d \sigma \hbox{ and so } \Vert f_{x} \Vert_{ \mathcal{H}^2} \leq \frac{\sqrt{3}  \Vert f \Vert_{ \mathcal{A}^2}}{2 x^3} \quad \forall x>0. \]
\end{rem}

Obviously we can do the same with the spaces $ \mathcal{A}_{ \mu}^{ 2}$.

\begin{cor}
Let $ \varepsilon >0$, we have: $T_{ \varepsilon}( \mathcal{A}_{ \mu}^2) \subset \mathcal{H}^2 \subset \mathcal{A}_{\mu}^2$.
\end{cor}


\subsection{Comparison between $ \mathcal{A}^p$ and $\mathcal{H}^p$.}
We already saw that $\mathcal{H}^p\subset\mathcal{A}^p$. The goal of this section is to prove the following theorem.

\begin{theoreme}{\label{injection2}}
Let $p > 2$. The identity from $ \mathcal{H}^2$ to $ \mathcal{A}^p$ is not bounded but the identity from $ \mathcal{H}^2$ to $ \mathcal{A}^2$ is compact.
\end{theoreme}

We need the following lemma (we did not find in the literature any such formula).

\begin{lemme}
For $n \geq 1$, we have:
$$\dis \sum_{k=0}^{+ \infty} \binom{n+k}{n}^2 z^k = \dis \frac{1}{(1-z)^{2n+1}}\,\sum_{k=0}^n \binom{n}{k}^2 z^k.$$
\end{lemme}

\begin{proof}
{Proof 1.} For $n=1$, we easily check that 
\[  \sum_{k=0}^{+ \infty} \binom{k+1}{1}^2 z^k = \frac{1+z}{(1-z)^{3}}. \]
We can now prove the equality by induction just by noting that 
\[ \binom{n+k}{n} \times \frac{(n+k+1)}{n+1} = \binom{n+k+1}{n+1}. \]
Now it suffices to compute the second derivative of the equality for the rank $n$ to obtain the equality for the rank $n+1$, nevertheless the computation is fastidious and we leave it to the reader.
\smallskip

{Proof 2.} We give also a quick and elementary argument. Fix $z\in{\mathbb D}$. We want to estimate $\dis S=\sum_{k=0}^\infty\binom{n+k}{k}^2 z^k.$

Since for every  $w\in{\mathbb D}$, we have $\dis\frac{1}{(1-w)^{n+1}}=\sum_{k=0}^\infty \binom{n+k}{k} w^k$, we point out that
$$\dis S=\frac{1}{n!}G^{(n)}(1)\quad\hbox{where}\quad G(w)=\frac{w^n}{(1-zw)^{n+1}}\cdot$$

Using now the Leibnitz formula, we get
$$S=\dis\frac{1}{n!} \sum_{k=0}^n\binom{n }{k}\; \frac{n!}{k!}\cdot \frac{(n+k)!z^{k}}{n!(1-z)^{n+k+1}}=\frac{1}{(1-z)^{2n+1}} \tilde S$$
where $\dis\tilde S=\sum_{k=0}^n \binom{n }{k}\binom{n+k}{k} z^{k}(1-z)^{n-k}$, which is the derivative of order $n$ at point $w=z$ of the function 
$$\dis w\longmapsto\dis\frac{1}{n!}\sum_{k=0}^n\binom{n }{k}  w^{n+k} (1-z)^{n-k}=\dis\frac{w^n}{n!}\big(w+1-z\big)^{n}.$$
Hence with help of  the Leibnitz formula once again, we obtain
 $$\dis\tilde S=\frac{1}{n!}\sum_{k=0}^n \binom{n }{k}\;  \frac{n!}{k!}z^k \cdot \frac{n!}{(n-k)!}(z+1-z)^{n-k}=\sum_{k=0}^n\binom{n }{k}^2 z^k.$$

We get the conclusion.
\end{proof}

After this work was completed, M. De La Salle communicated to us a third proof which relies on the computation of a residue. 

\begin{rem}
By uniqueness, we obtain 
$$\dis\sum_{j=0}^{ min(m, \, 2n+1)} (-1)^j \binom{2n+1}{j} \binom{n+m-j}{n}^2 = \binom{n}{m}^2 \quad \forall n, \, m \geq 1.$$

\end{rem}

\begin{defin}
Let $m \geq 1$ be an integer. We define $d_m$ as the following multiplicative function 
\[ d_m(k) = \biindice{\sum}{ d_1\ldots d_m= k}{d_1, \ldots, \, d_k \geq 1} 1. \]
\end{defin}

\begin{rem}
If we denote $*$ the Dirichlet convolution, $d_m$ is multiplicative because $d_m = \ind * \cdots * \ind$ ($m$ times) where $ \ind(n)=1$ for every $ n \geq 1$. 
\end{rem}

\begin{proposition}
Let $m \geq 1$ be an integer. There exists $\gamma_m>0$ such that
\[ \sum_{n=1}^{+ \infty} d_m(n)^2 n^{-2 \sigma} \sim \frac{\gamma_m}{(2 \sigma-1)^{m^2}} \quad \hbox{ when } \sigma  \rightarrow 1/2. \]
\end{proposition}

\begin{proof}
We know that $d_m$ is multiplicative, so we have 
\[  \sum_{n=1}^{+ \infty} d_m(n)^2 n^{-2 \sigma} = \prod_{p \in \mathbb{P}} \Bigg(\sum_{k \geq 0} d_m(p^k)^2 p^{-2 \sigma k} \Bigg). \]
Now, we can compute each series in the product because:
\[ d_m(p^k) =  \biindice{\sum}{p^{ \alpha_1} \, .. \, p^{ \alpha_m}=p^k}{\alpha_1, \dots, \, \alpha_m \geq 0} 1 = \biindice{\sum}{\alpha_1+ \dots + \alpha_m=k}{\alpha_1, \dots, \, \alpha_k \geq 0} 1 = \binom{m+k-1}{m-1}. \]
So, by the first lemma, we have 
\[  \sum_{n=1}^{+ \infty} d_m(n)^2 n^{-2 \sigma} = \prod_{p \in \mathbb{P}} \bigg{(} \frac{\sum_{k=0}^{m-1} \binom{m-1}{k}^2 (p^{-2 \sigma})^k}{(1-(p^{-2 \sigma})^k)^{2m-1}} \bigg{)} .\]
Now we have 
\[ \bigg{(}\sum_{k=0}^{m-1} \binom{m-1}{k}^2 z^k \bigg{)} \times (1-z)^{(m-1)^2} = Q(z) \]
where $Q(0)=1$ and $Q'(0) =0$ because the coefficient of $z$ is 
$$\dis\binom{m-1}{1}^2 - \binom{(m-1)^2}{1} = (m-1)^2-(m-1)^2=0.$$
We get that $(Q( p^{-2 \sigma}))_p \in \ell_1$ when $ \sigma \geq 1/2$ and the infinite product $ \prod_{p \in \mathbb{P}} Q(p^{-2 \sigma})$ is convergent and has a positive limit when $ \sigma \rightarrow 1/2$. Finally we obtain 
\[ \sum_{n=1}^{+ \infty} d_m(n)^2 n^{-2 \sigma} = \frac{\prod_{ p \in \mathbb{P}} Q(p^{-2 \sigma})}{ \prod_{p \in \mathbb{P}} (1-p^{-2 \sigma})^{(m-1)^2+2m-1}} = \frac{\prod_{ p \in \mathbb{P}} Q(p^{-2 \sigma})}{ \prod_{p \in \mathbb{P}} (1-p^{-2 \sigma})^{m^2}}. \]
And so when $ \sigma \rightarrow 1/2$, we obtain 
\[ \sum_{n=1}^{+ \infty} d_m(n)^2 n^{-2 \sigma} = \zeta(2 \sigma)^{m^2} \times  \bigg{(} \prod_{ p \in \mathbb{P}} Q(p^{-2 \sigma}) \bigg{)} \sim \frac{\gamma_m}{(2 \sigma-1)^{m^2}}. \]
\end{proof}

An immediate corollary is
\begin{cor}
Let $m \geq 1$ be an integer, there exists $c_m>0$ such that 
\[  \Vert \zeta^m( \sigma + \cdot ) \Vert_{ \mathcal{H}^2}  \sim \frac{c_m}{(2 \sigma-1)^{m^2/2}} \quad \hbox{ when } \sigma \rightarrow 1/2.\]
\end{cor}

Now we can prove Theorem \ref{injection2}.
\begin{proof}{\it of Th.\ref{injection2}.}
Assume that the injection from $ \mathcal{H}^2$ to $ \mathcal{A}^p$ is bounded, then there exists $m \geq 1$ such that 
$$\dis 2< \frac{2(m+1)}{m} <p.$$
The identity from $ \mathcal{H}^2$ to $ \mathcal{A}^{  \frac{2(m+1)}{m}}$ would be also bounded: there exists $C>0$ such that, for every $f \in \mathcal{H}^2$,
\[ \Vert f \Vert_{ \mathcal{A}^{  \frac{2(m+1)}{m}}} \leq C \Vert f \Vert_{ \mathcal{H}^2}. \]
We apply this inequality to the m-th power of the reproducing kernels of $ \mathcal{H}^2$: $s \rightarrow \zeta^m( \sigma +s)$ with $ \sigma> 1/2$. Then 
$$\dis\Vert \zeta^m( \sigma + \bullet ) \Vert_{ \mathcal{A}^{  \frac{2(m+1)}{m}}} \leq C \Vert \zeta^m( \sigma +\bullet ) \Vert_{ \mathcal{H}^2} $$

and thanks to the last corollary we know that 
\[ \Vert \zeta^m( \sigma + \bullet) \Vert_{ \mathcal{H}^2} \sim \frac{c_m}{(2 \sigma-1)^{m^2/2}} \quad \hbox{ when } \sigma \rightarrow 1/2. \]

Now for the first term in the inequality we have 
\[\Vert \zeta^m( \sigma + \bullet ) \Vert_{ \mathcal{A}^{  \frac{2(m+1)}{m}}} = \Vert \zeta^{m+1}( \sigma + \bullet) \Vert_{ \mathcal{A}^2}^{ \frac{m}{m+1}} = \bigg{(} \sum_{n=1}^{ + \infty} \frac{d_{m+1}(n)^2 n^{ -2 \sigma}}{ \log(n)+1} \bigg{)}^{  \frac{m}{2(m+1)}}. \]
By the previous proposition, we know that 
\[  \sum_{n=1}^{ + \infty} d_{m+1}(n)^2 n^{ -2 \sigma} \sim  \frac{\gamma_{m+1}}{(2 \sigma-1)^{(m+1)^2} } \quad \hbox{ when } \sigma  \rightarrow 1/2. \]
So by integration, we obtain 
\[ \bigg{(} \sum_{n=1}^{ + \infty} \frac{d_{m+1}(n)^2 n^{ -2 \sigma}}{ \log(n)+1} \bigg{)}^{  \frac{m}{2(m+1)}} \sim  \frac{\widetilde{\gamma_m}}{(2 \sigma-1)^{ \frac{m^2(m+2)}{2(m+1)}}}\]
for some $\dis\widetilde{\gamma_m}>0$.

Now using the inequality given by the boundedness of the identity, we obtain for $ \sigma$ close to $ 1/2$ 
\[1 \lesssim C \times  (2 \sigma-1)^{ \frac{m^2(m+2)}{2(m+1)}-m^2/2} = (2 \sigma-1)^{ \frac{m^2}{2(m+1)}} \]
and this is obviously false. \\

To finish the proof we have to show that the injection from $ \mathcal{H}^2$ to $ \mathcal{A}^2$ is compact but in fact it suffices to remark that this injection is a diagonal operator on the orthonormal canonical basis $\big(\e_n\big)_{n\ge1}$ of $ \mathcal{H}^2$: the eigenvalues, equal to  $\dis\frac{1}{\log(n)+1}\,$ ,  tends to zero.\end{proof}

\subsection{Inequalities on coefficients}

We shall give here some inequalities between the $ \mathcal{A}^p$ norm and some weighted $\ell^p$ norms of the coefficient of the functions. This follows the spirit of classical estimates on Bergman spaces (see \cite{duren2004bergman} p.81 for instance).

\begin{theoreme}{\label{theocoeffA}}
Let $p \geq 1$ and $\mu$ be a probability measure on $(0, + \infty)$ such that $0 \in Supp( \mu)$ and $(w_n)_{n\ge1}$ the associated weight. 
\begin{enumerate}[(i)]

\item When $1\le p\le2$ and $f=\dis\sum_{n\ge1} a_n \e_n\in \mathcal{A}_{\mu}^p$, we have 
$$\dis \Big\|w_n^{1/p} \, a_n\Big\|_{\ell^{p'}}\le\|f\|_{\mathcal{A}_{\mu}^p}. $$
\smallskip

\item When $p\ge2$ and $\dis\sum_{n\ge1} w_n^{p'-1} |a_n|^{p'}<\infty$, we have $f=\dis\sum_{n\ge1} a_n \e_n\in \mathcal{A}_{\mu}^p$ and
$$\dis \|f\|_{\mathcal{A}_{\mu}^p}\le\Bigg(\sum_{n\ge1}w_n^{p'-1} |a_n|^{p'}\Bigg)^{1/p'}=\Big\|w_n^{1/p} \, a_n\Big\|_{\ell^{p'}}. $$
\end{enumerate}

\end{theoreme}

An immediate corollary is 
\goodbreak

\begin{cor}
Let $p \geq 1$. 
\begin{enumerate}[(i)]

\item When $1\le p\le2$ and $f=\dis\sum_{n\ge1} a_n \e_n \in\mathcal{A}^p$, we have 
$$\dis \Big\|\frac{a_n}{(1+\ln(n))^{1/p}}\Big\|_{\ell^{p'}}\le\|f\|_{\mathcal{A}^p}. $$
\smallskip

\item When $p\ge2$ and $\dis\sum_{n\ge1}\frac{|a_n|^{p'}}{(1+\ln(n))^{p'-1}}<\infty$, we have $f=\dis\sum_{n\ge1} a_n \e_n\in \mathcal{A}^p$ and
$$\dis \|f\|_{\mathcal{A}^p}\le\Bigg(\sum_{n\ge1} \frac{|a_n|^{p'}}{(1+\ln(n))^{p'-1}}\Bigg)^{1/p'}=\Big\|\frac{a_n}{(1+\ln(n))^{1/p}}\Big\|_{\ell^{p'}}. $$
\end{enumerate}
\end{cor}

\begin{proof} {\it of Th.\ref{theocoeffA}}. 
Let us detail the case  $1\le p\le2$.

For every integer $n=p_1^{\alpha_1}p_2^{\alpha_2}\ldots\ge1$ and $f\in L^{p}({\mathbb R}^+\times{\mathbb T}^\infty,d\mu\otimes dm)$, let us define

$$\dis \tau_n(f)=\int_{{\mathbb R}^+\times{\mathbb T}^\infty}f(\sigma,z) \bar z^{(n)}n^{-\sigma}d\mu(\sigma)\otimes dm(z)$$
where $\dis z^{(n)}=z_1^{\alpha_1}z_2^{\alpha_2}\ldots$

We point out that, when $P$ is a Dirichlet polynomial $P(s)=\dis\sum_{n\ge1}a_n n^{-s}$, we can associate as usual $f(\sigma,z)=\dis\sum_{n\ge1}a_n n^{-\sigma}z_1^{\alpha_1}z_2^{\alpha_2}\ldots$ We have in that case $\dis\tau_n(f)=w_n a_n.$


Then we consider $\dis Q(f)=\big(\tau_n(f)\big)_{n\ge1}$.

$Q$ defines a norm one operator from $L^{1}({\mathbb R}^+\times{\mathbb T}^\infty,d\mu\otimes dm)$ to $L^\infty(\omega)$ and from $L^{2}({\mathbb R}^+\times{\mathbb T}^\infty,d\mu\otimes dm)$ to $L^2(\omega)$, where $L^q(\omega)$ is the Lebesgue space on positive integers with discrete measure whose mass at point $n$ is given by $1/w_n$. Indeed:

$$\dis|\tau_n(f)|\le\|f\|_1$$
and 
$$\dis\sum_{n\ge1}\frac{|\tau_n(f)|^2}{w_n}=\sum_{n\ge1}|\langle f, b_n\rangle|^2$$
where $\dis b_n(\sigma,z)= z^{(n)}n^{-\sigma}/\sqrt{w_n}$ is an orthonormal system in the Hilbert space ${L^{2}({\mathbb R}^+\times{\mathbb T}^\infty,d\mu\otimes dm)}$. So the Bessel inequality gives
$$\dis\sum_{n\ge1}\frac{|\tau_n(f)|^2}{w_n}\le\|f\|_2^2. $$


Now by interpolation (apply the Riesz-Thorin theorem), $Q$ is bounded from  $L^{p}({\mathbb R}^+\times{\mathbb T}^\infty,d\mu\otimes dm)$ to $L^{p'}(\omega)$:

$$\dis\Bigg(\sum_{n\ge1}\frac{|\tau_n(f)|^{p'}}{w_n}\Bigg)^{1/p'}\le\|f\|_p\;.$$
Writing this inequality  in the particular case of $f$ associated to a Dirichlet polynomial (as described as the beginning of the proof), the result follows.

The other case is obtained (it is even easier) in the same way.
\end{proof}


\section{The Bergman spaces $ \mathcal{B}^p$}
\subsection{The Bergman spaces of the infinite polydisk}
Recall that $A= \lambda \otimes \lambda \otimes \cdots$ where $ \lambda$ is the normalized Lebesgue measure on the unit disk $ \mathbb{D}$. For $p \geq 1$, $B^p( \mathbb{D}^{ \infty})$ is the Bergman space of the infinite polydisk. It is defined as the closure in the space $L^p( \mathbb{D}^{ \infty}, A)$ of the span of the analytic polynomials. \\

\begin{rem}
Let $P$ be an analytic polynomial defined for $z=(z_1,z_2, \, \dots) \in\mathbb{D}^{ \infty}$ by $\dis P(z) := \sum_{n=1}^{N} a_n z_1^{ \alpha_1} \cdot \cdot \, \, z_k^{ \alpha_k}$. Then 
$$\dis \Vert P \Vert_{B^2} = \bigg{(} \sum_{n=1}^{ N} \frac{\vert a_n \vert^2}{(\alpha_1+1) \cdot \cdot \, \, (\alpha_k+1)} \bigg{)}^{1/2}. $$

So clearly $H^2( \mathbb{T}^{ \infty})\subset B^2( \mathbb{D}^{ \infty})$. In fact this is also true for any $p \geq 1$, it suffices to apply several times this property in the case of the unit disk. 
\end{rem}

Recall that the Bergman kernel at $z,w \in \mathbb{D}$ is defined by $k(w,z):= \displaystyle{\frac{1}{(1- \overline{w}z)^2}}$.

\begin{defin}
Let $z \in \mathbb{D}^{ \infty}$ and $ \zeta \in \mathbb{D}^{ \infty} \cap \ell_2$. For $n \geq 1$, we define 
\[ K_n(\zeta,z) := \prod_{i=1}^n k(\zeta_i,z_i) \hbox{ and } K(\zeta,z) := \prod_{i=1}^{+ \infty} k(\zeta_i,z_i). \]
$K$ is well defined thanks to the condition on $\zeta$ and the fact that $(K_n)$ converges pointwise to $K$.
\end{defin}

\begin{rem}
We know that 
\[ \Vert k( \zeta_i, \, \bullet \, ) \Vert_{2}^{2} = k( \zeta_i,\zeta_i) = \frac{1}{(1-\vert \zeta_i \vert^2)^2}.\]
So $K( \zeta, \, \cdot \,) \in B^2( \mathbb{D}^{ \infty})$ and 
\[ \Vert K( \zeta, \,  \bullet \, ) \Vert_{B^2( \mathbb{D}^{\infty})}^{2} = \prod_{i=1}^{+ \infty}  \frac{1}{(1-\vert \zeta_i \vert^2)^2}. \]
\end{rem}

\begin{proposition}
Let $P$ be an analytic polynomial on $ \mathbb{D}^{ \infty}$ and let $ \zeta \in \mathbb{D}^{ \infty} \cap \ell_2$. Then 
\[ \vert P( \zeta ) \vert \leq \bigg{(} \prod_{i=1}^{+ \infty}  \frac{1}{1-\vert \zeta_i \vert^2} \bigg{)} \Vert P \Vert_{B^2( \mathbb{D}^{\infty})}. \]
\end{proposition}

\begin{proof}
By the reproducing kernel property of the classical Bergman space used several times, we obtain that:
\[ P(\zeta_1,\ldots,\zeta_n) = \int_{ \mathbb{D}^n} P(z_1,\ldots,z_n) \overline{K_n(\zeta,z)} d \lambda(z_1).. d \lambda(z_n). \]
The Cauchy-Schwarz inequality gives the result. 
\end{proof}

With the previous proposition, we can extend by density the evaluation defined on the analytic polynomials for $z \in \mathbb{D}^{ \infty} \cap \ell_2$. For every $f \in B^2( \mathbb{D}^{ \infty})$, we denote $ \tilde{f}( \zeta)$ this extension and we have 
\[ \vert \tilde{f}( \zeta ) \vert \leq \bigg{(} \prod_{i=1}^{+ \infty}  \frac{1}{1-\vert \zeta_i \vert^2} \bigg{)} \Vert f \Vert_{B^2( \mathbb{D}^{\infty})}. \]
Moreover the norm of the evaluation is exactly $\displaystyle{\prod_{i=1}^{+ \infty}}  \frac{1}{1-\vert \zeta_i \vert^2}$. Actually in \cite{cole1986representing}, the authors proved (in a more general setting) that $ \tilde{f}$  is holomorphic on $ \mathbb{D}^{ \infty} \cap \ell_2$. We shall need the next lemma.\\

\begin{lem}{(\cite{cole1986representing})}
Let $ \zeta \in  \mathbb{D}^{ \infty} \cap \ell_2$, $N\ge1$ and $a \in \mathbb{R}$, we set 
\[ G_N(z) := \prod_{i=1}^N (1- \overline{\zeta_j} z_j)^a .\]
Then  $ \lbrace G_N \rbrace$ is a bounded martingale in $ L^2( \mathbb{T}^{ \infty})$.
\end{lem}

\begin{rem}
Each element $G_n$ belongs to $B^2( \mathbb{D}^{ \infty})$ and we have 
\[ \Vert G_n \Vert_{ L^2( \mathbb{D}^{ \infty})} =  \Vert G_n \Vert_{ B^2( \mathbb{D}^{ \infty})} \leq  \Vert G_n \Vert_{ H^2( \mathbb{T}^{ \infty})} = \Vert G_n \Vert_{ L^2( \mathbb{T}^{ \infty})} .\]
So $ \lbrace G_n \rbrace$ is a bounded martingale in $L^2 ( \mathbb{D}^{ \infty})$. By the Doob's theorem and by closure we know that the product converges pointwise and in norm in $ B^2( \mathbb{D}^{ \infty})$. 
\end{rem}

We need to recall some notations and results from \cite{cole1986representing}. \\

Let $U$ be a uniform algebra on a compact space $X$ and $ \mu$ be a measure on $X$. $H^p( \mu)$ is the closure of $U$ in $ L^p( \mu)$.
 
\begin{prop}{(\cite{cole1986representing})}
Let $U$ be a uniform algebra on a compact space $X$, $ \mu$ be a probability measure on $X$, $y \in X$ such that the evaluation extends continuously to $H^2( \mu)$. Assume that any real power of the reproducing kernel of this point evaluation $ x \rightarrow K(x,y)$ belongs to $H^2( \mu)$. Then for $p \geq 1$, we have 
\[ \vert \tilde{f}(y) \vert^p \leq K(y,y) \, \int \vert f(x) \vert^p d \mu(x) \] 
for every function $f$ in $H^p( \mu)$ and the norm of the point evaluation at $y$ is exactly $K(y,y)^{1/p}$.
\end{prop}

With the last remark and this theorem, we obtain that the point evaluation is bounded on $B^p( \mathbb{D}^{ \infty})$ for $ \zeta \in \mathbb{D}^{ \infty} \cap \ell_2$ and we have 
\[ \vert f( \zeta) \vert^p \leq  \prod_{i=1}^{+ \infty}  \, \, \frac{1}{(1-\vert \zeta_i \vert^2)^2} \Vert f \Vert_{B^p( \mathbb{D}^{ \infty})}^p. \]

Moreover $ \tilde{f}$ is holomorphic on $ \mathbb{D}^{ \infty} \cap \ell_2$ thanks to \cite{cole1986representing}.

\subsection{Point evaluation on $ \mathcal{B}^p$}


In the sequel, $R$ will denote the infinite product of the probability measures $2r_i dr_i$ on $[0,1]$.

\begin{defin}
Let $P \in \mathcal{P}$ of the form $\sum^N_{n \geq 1} a_n n^{-s}$. We define on $ \mathcal{P}$ the norm
\[ \Vert P \Vert_{ \mathcal{B}^p} := \bigg{(} \int_{ [0,1]^{ \infty}} \bigg{\Vert} \sum_{n=1}^{N} a_n r_1^{ \alpha_1} \cdot \cdot \, \,  r_k^{ \alpha_k} \e_n(it) \bigg{\Vert}_{ { \mathcal{H}^p}}^p\,  dR \bigg{)}^{1/p}. \]
\end{defin}

\begin{rem}
The fact that this defines a norm follows from the next proposition.
\end{rem}

\begin{defin}
Let $p \geq 1$. We denote by $\mathcal{B}^p$ the closure of $ \mathcal{P}$ relatively to the norm $\Vert\cdot \Vert_{ \mathcal{B}^p}$: it is the Bergman space of Dirichlet series.
\end{defin}

\begin{rem}
We denote $d(n)$ the number of divisors of $n$. For $f$ as in $(1)$, one has 
\[ \Vert f \Vert_{ \mathcal{B}^2} = \bigg{(} \sum_{n=1}^{+ \infty} \frac{\vert a_n \vert^2}{d(n)} \bigg{)}^{1/2}\,. \]
\end{rem}

First we use the Bohr's point of view to precise the link between $ \mathcal{B}^p$ and $B^p( \mathbb{D}^{ \infty})$.

\begin{proposition}
Let $p \geq 1$.
\begin{enumerate}[(i)]
\item Let $P \in \mathcal{P}$. We have $\dis\Vert P \Vert_{ \mathcal{B}^p} = \Vert D(P) \Vert_{B^p}$. 

\item $\dis D : \mathcal{P} \rightarrow B^p( \mathbb{D}^{\infty})$ extends to an isometric isomorphism from $ \mathcal{B}^p$ onto $ B^p( \mathbb{D}^{\infty})$.
\end{enumerate}
\end{proposition}

\begin{proof} The first fact is clear. For the second one, remember that $\mathcal{B}^p$ is the closure of $ \mathcal{P}$ and that $B^p( \mathbb{D}^{\infty})$ is the closure of the set of analytic polynomials. \end{proof}

\begin{theoreme}{\label{evalBp}}
Let $p \geq 1$ and $f \in \mathcal{B}^p$. The abscissa of uniform convergence of $f$ verifies $\sigma_u(f) \leq \frac{1}{2}$. Moreover, when $\Re(w) > \frac{1}{2}$, we have 
\[ \vert f(w) \vert \leq \zeta(2 Re(w))^ {2/p}\, \Vert f \Vert_{ \mathcal{B}^p} \]
and:
\[ \Vert \delta_w \Vert = \zeta(2Re(w))^{2/p}. \]
In addition, there exists $f \in \mathcal{B}^p$ such that $ \sigma_{b}(f)= 1/2$.
\end{theoreme}

\begin{proof}
Let $f \in \mathcal{B}^p$ and $s \in \mathbb{C}_{ \frac{1}{2}}$. We define $z_s=(p_1^{-s},p_2^{-s},\ldots) \in \mathbb{D}^{ \infty} \cap \ell_2$. We know that $D(f) \in B^p( \mathbb{D}^{ \infty})$ and so 
\[ \vert D(f)(z_s) \vert^p \leq \prod_{i=1}^{ + \infty} \frac{1}{(1- \vert p_i^{-s} \vert^2)^2} \, \Vert D(f) \Vert_{ B^p( \mathbb{D}^{ \infty})}^p. \]
But thanks to the last proposition $\Vert D(f) \Vert_{ B^p( \mathbb{D}^{ \infty})} = \Vert f \Vert_{ \mathcal{B}^p}$ and 

$$
\begin{array}{cl}
D(f)(z_s) &\dis=\biindice{\sum}{n=p_1^{ \alpha_1} \ldots \, \, p_k^{ \alpha_k}}{ n\geq 1} a_n (p_1^{-s})^{ \alpha_1}\ldots \, \, (p_k^{-s})^{ \alpha_k} \\ \\

&\dis= \biindice{\sum}{n=p_1^{ \alpha_1} \ldots \, \, p_k^{ \alpha_k}}{ n\geq 1} a_n (p_1^{ \alpha_1} \ldots \, \, p_k^{ \alpha_k})^{-s}  = \sum_{n \geq 1} a_n n^{-s}.\\
\end{array}
$$
Then we have 
$$\dis\vert f(s) \vert^p \leq \prod_{i=1}^{ + \infty} \frac{1}{\Big(1- p_i^{-2 \Re(s)}\Big)^2} \, \Vert f \Vert_{ \mathcal{B}^p} = \zeta(2 \Re(s))^2 \, \Vert f \Vert_{ \mathcal{B}^p}^p. $$ 
So $f$ admits a bounded extension on each smaller half-plane of $ \mathbb{C}_{ \frac{1}{2}}$. By the Bohr's theorem we have $ \sigma_u(f) \leq \frac{1}{2}$. \\

To prove that the norm of the evaluation is exactly $\zeta(2Re(w))^{2/p}$, it suffices to use the corresponding result from \cite{cole1986representing} on $B^p( \mathbb{D}^{ \infty})$. 
\end{proof}

\subsection{Comparison between $ \mathcal{B}^p$ and $ \mathcal{H}^p$}

In this section, we precise the link between $ \mathcal{B}^p$ and $ \mathcal{H}^p$. This question is natural as soon as we keep in mind the behavior of the injection from $H^p( \mathbb{D})$ to $B^{q}( \mathbb{D})$ in the classical framework of one variable Hardy-Bergman spaces on the unit disk. We recall that $H^p\subset B^q$ if and only if $q\le 2p$ and that this injection is compact if and only if $q<2p$ (see \cite{LRP} for recent results on the limit case $q=2p$).

First, following ideas of \cite{bayart2002hardy}, we obtain a result of hypercontractivity between the spaces $ \mathcal{B}^p$. \\

Let $1 \leq p \leq q < + \infty$. For $f \in B^p( \mathbb{D}^{ \infty})$, $z \in \mathbb{D}^{ \infty}$ and $k \geq 1$, we define $ \hat{z_k} =( z_1, \, \dots \,  ,z_{k-1},z_{k+1}, \, \dots)$. Let $f_{ \hat{z_k}}(z_k) = f(z)$. Then 
\[ \int_{  \mathbb{D}^{ \infty}}  \Vert f_{ \hat{z_k}} \Vert_{ L^r ( \mathbb{D})}^r dm( \hat{z_k}) = \Vert f \Vert_{L^r ( \mathbb{D}^\infty)}^r. \]

We consider a sequence of operators $S_k : B^{p}( \mathbb{D}) \rightarrow B^{q} ( \mathbb{D})$ for $k \geq 1$ such that $S_k(1)=1$. If $P$ is an analytic polynomial on $ \mathbb{D}^{ \infty}$, we define 
\[\left\lbrace 
	\begin{array}{c}
	P^1 = P,  \\
 	P^{k+1}= S_k( P_{ \hat{z_k}} (z_k)). \end{array} \right. \]

This sequence is not in general a sequence of polynomials but if $P$ depends of $z_1, \, \dots  , \, z_n$ then each term of this sequence too. So this sequence is stationary.

\begin{proposition}\label{HYperC}
If $ \prod_{k=1}^{ + \infty} \Vert S_k \Vert < + \infty$ then the sequence $(P^k)_{k\ge1}$ converges to some $S(P)$ of $B^q( \mathbb{D}^{ \infty})$. In addition, $S$ extends to a bounded operator from $B^p( \mathbb{D}^{ \infty})$ to $B^q( \mathbb{D}^{ \infty})$.
\end{proposition}

Actually, if we consider a sequence of operators $S_k : H^{p}( \mathbb{D}) \rightarrow B^{q} ( \mathbb{D})$, we obtain the following similar result.

\begin{proposition}{\label{hypercont}}
If $ \prod_{k=1}^{ + \infty} \Vert S_k \Vert < + \infty$ then the sequence $(P^k)$ converges to some $S(P)$ of $B^q( \mathbb{D}^{ \infty})$. In addition, $S$ extends to a bounded operator from $H^p( \mathbb{T}^{ \infty})$ to $B^q( \mathbb{D}^{ \infty})$.
\end{proposition}

We only give the proof of the second proposition.

\begin{proof}
It suffices to show that 
\[ \Vert P^{(k+1)} \Vert_{ B^q( \mathbb{D}^{ \infty})} \leq \bigg{(} \prod_{i=1}^{k} \Vert S_i \Vert \bigg{)} \, \Vert P \Vert_{ H^p( \mathbb{T}^{ \infty})}. \]
One has
$$\begin{array}{ll}
\dis\Vert P^{(k+1)} \Vert_{ B^q( \mathbb{D}^{ \infty})}^q &\dis= \int_{ \mathbb{D}^{ \infty}} \vert S_k( P_{ \hat{z_k}}^k (z_k)) \vert^q dA(z)\\ \\

   &\dis= \int_{ \mathbb{D}^{ \infty}} \int_{ \mathbb{D}} \vert S_k(P_{ \hat{z_k}}^k (z_k)) \vert^q\, d\lambda(\hat{z_k}) dA(z_k)\\ \\
\dis &\dis= \int_{ \mathbb{D}^{ \infty}} \Vert S_k ( P_{ \hat{z_k}}^k(.)) \Vert_{ B^q( \mathbb{D})}^q dA(\hat{z_k})\\
 &\dis \leq \Vert S_k \Vert^q \, \int_{ \mathbb{D}^{ \infty}} \Vert P_{ \hat{z_k}}^k(\cdot) \Vert_{ H^p( \mathbb{T})}^q d A(\hat{z_k}) \\
 &\dis =  \Vert S_k \Vert^q \, \int_{ \mathbb{D}^{ \infty}} \bigg{(} \int_{ \mathbb{T}} \vert P_{ \hat{z_k}}^k(\chi_k) \vert^p dm( \chi_k) \bigg{)}^{q/p} \, dA(\hat{z_k}). \\
\end{array}
$$


Since $\dis\frac{q}{p}\ge1$, we get, by the integral triangular inequality,
\[ \Vert P^{(k+1)} \Vert_{ B^q( \mathbb{D}^{ \infty})}^q \leq \Vert S_k \Vert^q \bigg{(} \int_{ \mathbb{T}} \bigg{(} \int_{\mathbb{D}^{ \infty}} \vert P_{ \hat{z_k}}^k(\chi_k) \vert^q dm(\hat{z_k}) \bigg{)}^{ p/q} dm( \chi_k) \bigg{)}^{q/p} \]

By induction, we obtain the result. 
\end{proof}

We shall give some applications of these propositions, but we first need other preliminaries, in the classical setting of the unit disk.In the following, for $q\ge1$, the space $B^q( \mathbb{D})$ (resp. $H^q( \mathbb{D})$) is the classical Bergman space (resp. the classical Hardy space).

\begin{lemme}
The sequence $\dis\Big(\frac{2}{n+2}\Big)_{n \geq 0}$ defines a multiplier from $B^1( \mathbb{D})$ to $H^1( \mathbb{D})$ with norm exactly equal to $1$: for every $f(z)=\dis\sum_{n\ge1}a_nz^n\in B^1( \mathbb{D})$, we have 
$$\Big\|\sum_{n\ge1}\frac{2}{n+2}a_nz^n\Big\|_{H^1( \mathbb{D})}\le\|f\|_{B^1( \mathbb{D})}. $$
\end{lemme}

\begin{proof}
Let $r<1$ and $f \in B^1( \mathbb{D})$ of the form $ f(z) = \displaystyle{\sum_{n=0}^{ + \infty}} a_n z^n$. Then if we denote by $M$ this multiplier operator, we have 
\[ \frac{1}{2 \pi} \int_{0}^{ 2 \pi} \vert Mf(r e^{i \theta}) \vert \, d \theta = \frac{1}{2 \pi} \int_{0}^{2 \pi} \bigg{|} \sum_{n=0}^{+ \infty} \frac{2}{n+2} a_n  r^n e^{in \theta} \bigg{|} d \theta \]
\[ =  \frac{1}{2 \pi} \int_{0}^{ 2 \pi} \bigg{|} \sum_{n=0}^{+ \infty} 2 \int_{0}^{1}  a_n  \rho^{n+1} r^n e^{in \theta} \, d \rho \bigg{|} d \theta \leq  \frac{1}{ \pi} \int_{0}^{ 2 \pi}  \int_{0}^{1} \bigg{|} \sum_{n=0}^{+ \infty}   a_n  \rho^{n} r^n e^{in \theta}  \bigg{|} \, \rho \, d \rho d \theta. \]
If $r$ goes to $1$, we obtain the result.
\end{proof}

\begin{lemme}
Let $r \leq \frac{2}{3}$. Then $\Big(r^n \frac{n+2}{2\sqrt{n+1}}\Big)_{n \geq 0}$ is a multiplier from $H^1( \mathbb{D})$ to $H^2( \mathbb{D})$ with norm $1$.
\end{lemme}

\begin{proof}
We adapt a proof from \cite{bonami1970etude}. Let $f \in H^1( \mathbb{D})$, with norm $1$, of the form 
\[ f(z) = \sum_{n=0}^{ + \infty} a_n z^n .\]
We considerer the factorisation $f=gh$ where $g$ and $h$ are in $H^2( \mathbb{D})$ and verify $ \vert g \vert^2 = \vert h \vert^2 = 1$. Denote $(b_n)$ and $(c_n)$ the Fourier coefficients of $g$ and $h$. Then we have:
\[ a_n = \sum_{k=0}^n b_k c_{n-k}. \]
We also know 
\[ \sum_{n=0}^{ + \infty} \vert b_n \vert^2 = \sum_{n=0}^{ + \infty} \vert c_n \vert^2 = 1. \]
We want to show that 
\[ \sum_{n=0}^{ + \infty} \bigg{\vert} a_n r^n \frac{n+2}{2\sqrt{n+1}} \bigg{\vert}^2 \leq 1. \]
We can assume that the coefficients $b_n$ et $c_n$ are all non negative (At worst, the modulus of $a_n$ become bigger but we search a sufficient condition for the inequality so this is not a problem). \\

So the last inequality is equivalent to 
\[ \sum_{n=0}^{+\infty} \sum_{k=0}^n b_k c_{n-k} r^n  d_n \frac{n+2}{2\sqrt{n+1}} \leq 1 \] 
for all non negative sequence $(d_n)$ with $\ell_2$-norm 1. This is equivalent to 
\[ \sum_{k,n=0}^{+ \infty} b_k c_n r^{n+k} d_{n+k}   \frac{n+k+2}{2\sqrt{n+k+1}} \leq 1.\]
We will get this inequality as soon as 
\[ \sum_{k,n=0}^{ + \infty} \bigg{(} r^{n+k} \frac{n+k+2}{2\sqrt{n+k+1}} d_{n+k} \bigg{)}^2  \leq 1. \]
But for any $j$, we only have $j+1$ ways to write this integer as the sum of two integers. So it suffices to prove the following inequality 
$$\dis\sum_{j=0}^{ + \infty} r^{2j} \frac{(j+2)^2}{4} d_{j}^2  \leq 1.$$
And with the definition of $(d_n)$, it will be true as soon as 
\[  r^{2j} \frac{(j+2)^2}{4} \leq 1 \quad \hbox{ for any } j \geq 0. \]
This latter inequality is clearly true for $j =0$ for any $r$, so we just have to compute 
\[ r_0 = \inf_{j \geq 1} \bigg{(} \frac{2}{j+2} \bigg{)}^{1/j}. \]
We can check easily that $x \rightarrow \frac{ln(2/(x+2))}{x}$ is increasing on $[1, + \infty[$ so we obtain $\dis r_0 = \frac{2}{3}$. \end{proof}

The following is obvious and is just a rewritting of the norms.
\begin{lemme}
$(\sqrt{n+1})_{n \geq 0}$ is a multiplier from $H^2( \mathbb{D})$ to $B^2( \mathbb{D})$ with norm exactly equal to $1$.
\end{lemme}

Now we can state a contractive type result on classical Bergman spaces. Here $P_r$ denotes the blow-up operator: $P_r(f)(z)=f(rz)$.

\begin{theoreme}
If $r \leq \displaystyle{\frac{2}{3}}$, $P_r : B^1( \mathbb{D}) \rightarrow B^2(\mathbb{D})$ is bounded with norm 1. Conversely, if $P_r : B^1( \mathbb{D}) \rightarrow B^2(\mathbb{D})$ is bounded with norm 1 then $r \leq \dis\frac{1}{\sqrt{2}}$.
\end{theoreme}

\begin{proof}
If $r \leq \frac{2}{3}$, it suffices to apply the three previous lemmas.

Conversely assume that $P_r : B^1( \mathbb{D}) \rightarrow B^2(\mathbb{D})$ is bounded with norm 1. Let $a \in \mathbb{R}$, we have 
\[ \Vert 1 +arz \Vert_{B^2( \mathbb{D})}^2 = 1 + \frac{a^2r^2}{2} \]
\[ \Vert 1+az \Vert_{B^1( \mathbb{D})} = 1 + \frac{a^2}{8} + \circ(a^2). \]
So we have 
\[ 1 + \frac{a^2r^2}{2} \leq \bigg{(} 1 + \frac{a^2}{8} + \circ(a^2) \bigg{)}^2 = 1 + \frac{a^2}{4} + \circ(a^2). \]
And so $r^2 \leq \frac{1}{2}$. \end{proof}

Now we have another consequence of the preceding results, which will be used in the next section, and is similar to Prop.\ref{TepsPA}.

\begin{proposition}{\label{Tepsberg}}
Let $ \varepsilon > 0$. Then $T_{ \varepsilon} : \mathcal{B}^1 \rightarrow \mathcal{B}^2$ is bounded.
\end{proposition}

\begin{proof}
We consider the following sequence of operators (we keep the notations of the preceding theorem)
$$\begin{array}{cccc}
S_k \, : \,  & B^1( \mathbb{D}) & \longrightarrow & B^2( \mathbb{D}), \\
           & f                & \longmapsto & P_{p_k^{ - \varepsilon}}(f) \\
\end{array}$$

where $P_r$ is the classical Poisson kernel. Indeed if we apply  Prop.\ref{HYperC} to this sequence of operators and to a Dirichlet series $f$ of the form $(1)$, we obtain 
$$\begin{array}{cl}
 S(f)(s) &\dis= \sum_{n=p_1^{ \alpha_1} \cdot \cdot \, \, p_k^{ \alpha_k} \geq 1} a_n (p_1^{ - s - \varepsilon})^{ \alpha_1} \cdot \cdot \, \, (p_k^{ - s - \varepsilon})^{ \alpha_k}\\
&\dis =\sum_{n=p_1^{ \alpha_1} \cdot \cdot \, \, p_k^{ \alpha_k} \geq 1} a_n n^{-s - \varepsilon} = T_{ \varepsilon}(f)(s). \\
\end{array}$$

We know from the preceding theorem that $ \Vert P_r \Vert_{ B^1(  \mathbb{D}) \rightarrow B^2( \mathbb{D})} \leq 1$ for $r$ quite small and we obtain our result for $T_{ \varepsilon}$ because $p_k^{ - \varepsilon} \rightarrow 0$ when $k$ goes to infinity and so the infinite product of the norm is finite. 
\end{proof}

\begin{notas}
Let $p \geq 1$. We denote $ \mathcal{H}_{ \mathbb{P}}^p$ (resp. $\mathcal{B}_{ \mathbb{P}}^p)$) the following subspace of $\mathcal{H}^p$ (resp. $\mathcal{B}^p$):
\[ \mathcal{H}_{ \mathbb{P}}^p = \overline{span( e_{k}, \, k\in\mathbb{P})}^{ \mathcal{H}^p} \, \big{(}resp. \, \, \mathcal{B}_{ \mathbb{P}}^p = \overline{span( e_{k}, \, k\in\mathbb{P})}^{ \mathcal{B}^p} \big{)}.\]
\end{notas}

\begin{theoreme}{\label{injection1}}
Let $p \geq 1$.
\begin{enumerate}[(i)]
\item
The identity from $\mathcal{H}^p$ to $ \mathcal{B}^{2p}$  is bounded with norm $1$.

\noindent But,
\item
the identity from $\mathcal{H}^p$ to $ \mathcal{B}^p$ is not compact. Actually, it is not a strictly singular operator.
\end{enumerate}
\end{theoreme}

\begin{proof} $(i)$ Recall (\cite{duren2004bergman}) that  the identity from $H^p(\mathbb{D})$ to $B^{2p}( \mathbb{D})$ is bounded with norm $1$ so it suffices to use Prop.\ref{hypercont} to get the boundedness of our operator.

$(ii)$ In \cite{bayart2002hardy}, it is shown that $\mathcal{H}_{ \mathbb{P}}^p= \mathcal{H}_{ \mathbb{P}}^2$ and by the same way we obtain that $\mathcal{B}_{ \mathbb{P}}^p= \mathcal{B}_{ \mathbb{P}}^2$. But clearly $\mathcal{H}_{ \mathbb{P}}^2 = \mathcal{B}_{\mathbb{P}}^2$ so $\mathcal{H}^p$ and $\mathcal{B}^p$ have isomorphic infinite-dimensional closed subspaces and so the identity from $\mathcal{H}^p$ to $\mathcal{B}^p$ is not strictly singular.
\end{proof}

\begin{rems}
$\hbox{  }$

\begin{enumerate}[(i)]
\item
When $p=1$, $(i)$ has already been proved by Helson (see \cite{helson2005dirichlet}).
\item
We can check easily that for every $n \neq m$ we have:
\[ \Vert \e_{p_n} - \e_{p_m} \Vert_{ \mathcal{B}^p} \geq \Vert \e_{p_n} \Vert_{ \mathcal{B}^p} = \bigg{(} \frac{2}{p+2} \bigg{)}^{1/p} \]
and then we obtain another proof of the non compactness in Th.\ref{injection1}(ii).
\item
Let us mention that it is immediate (without invoking Th.\ref{injection1}(ii)) that the identity from $\mathcal{H}^p$ to $\mathcal{B}^{2p}$ is not compact: indeed if it were,  by restriction to the variable $z_1=2^{-s}$, the identity from $H^p(  \mathbb{D})$ to $B^{2p}( \mathbb{D})$ would be compact but this is not the case.
\end{enumerate}
\end{rems}

Actually, we can prove that $\dis\mathcal{H}^2\subset\mathcal{B}^4$ by a simple computation on the coefficients of the Dirichlet series. Let $f$ be a Dirichlet series of the form $(1)$. We want to show that $\Vert f \Vert_{ \mathcal{B}^4} \leq \Vert f \Vert_{ \mathcal{H}^2}$. We have 
\[ \Vert f \Vert_{ \mathcal{B}^4}^4 = \Vert f^2 \Vert_{ \mathcal{B}^2}^{2} .\]
But $f^2(s) =\dis \sum_{n=1}^{ + \infty} b_n n^{ -s}$ with 
\[ b_n = \sum_{ d | n} a_d \times  a_{n/d} \, \; \, \qquad \forall n \geq 1. \]
So 
\[  \Vert f^2 \Vert_{ \mathcal{B}^2}^{2} = \sum_{n=1}^{ + \infty} \frac{\vert \sum_{ d | n} a_d \times  a_{n/d} \vert^2}{d(n)}. \]
Now we apply the Cauchy-Schwarz inequality using the fact that the sum contains exactly $d(n)$ terms 
\[ \leq  \sum_{n=1}^{ + \infty} \sum_{d | n} \vert a_d \vert^2\, \vert a_{n/d} \vert^2 .\]
We have $n \geq1$ and $\dis n= d \times \frac{n}{d}$, then we can exchange the sums  
\[ = \sum_{d=1}^{ + \infty} \vert a_d \vert^2 \, \sum_{ d|n } \vert a_{ n/d} \vert^2. \]
But if $n$ is a multiple of $d$, then $\dis\frac{n}{d}$ is in $ \mathbb{N}$, and 
\[   \sum_{ d|n } \vert a_{ n/d} \vert^2 = \Vert f \Vert_{ \mathcal{H}^2}^2. \]
At last, we get
\[ \Vert f \Vert_{ \mathcal{B}^4}^4 \leq  \Vert f \Vert_{ \mathcal{H}^2}^4.  \]

\subsection{Generalized vertical limit functions.}
\begin{defin}
Let $\chi \in \mathbb{D}^{ \infty}$ and $f$ of the form $(1)$. We denote by $f_{ \chi}$  the following Dirichlet series:
\[ f_{ \chi}(s) = \sum_{n=1}^{ + \infty} a_n \chi(n) n^{-s}. \]
\end{defin}

In this part, we apply the same trick than in \cite{gordon1999composition} and \cite{bayart2002hardy} to obtain another expression of the norm in $ \mathcal{B}^p$ useful for the study of composition operators. \\

Let $ \varphi_1(z)= \frac{1+z}{1-z}$ the Cayley transform which maps $\mathbb{D}$ on $ \mathbb{C}_+$. We will say that a function is in $H^{p}_i( \mathbb{C}_+)$ if $f \circ \varphi_1 \in H^p( \mathbb{D})$ (The classical Hardy space of the unit disk). In this case we have 
\[ \frac{1}{2 \pi} \int_{- \pi}^{ \pi} \vert f \circ \varphi_1 (e^{ i \theta}) \vert^p d \theta = \int_{ \mathbb{R}} \vert f(it) \vert^p d \lambda_i(t) \]
where
\[d \lambda_i(t) = \frac{dt}{\pi(1+ t^2)} . \]

\begin{defin}
Let $t \in \mathbb{R}$. We define the Kronecker flow $T_t$ on $  \mathbb{D}^{ \infty}$ by  $T_t( z_1, z_2, \,
 \dots) := (p_1^{-it} z_1, p_2^{-it}z_2, \, \dots)$.
\end{defin}

\begin{lemme}
Let $ \chi \in\mathbb{D}^{ \infty}$, $f \in \mathcal{B}^p$ and $t \in \mathbb{R}$. We set $ g_{ \chi} (it) := D(f) ( T_t \chi )$.
Then for $w$ a finite Borel measure on $ \mathbb{R}$, one has 
\[ \int_{\mathbb{D}^{ \infty}} \int_{ \mathbb{R}} \vert g_{ \chi} (it) \vert^p dw(t) dm( \chi) = w( \mathbb{R}) \Vert f \Vert_{ \mathcal{B}^p}^p. \]
\end{lemme}

\begin{proof}
The Kronecker flow $(T_t)$ is just a rotation on $ \mathbb{D}^{ \infty}$, so 
\[  \int_{ \mathbb{D}^{ \infty}} \vert g_{ \chi} (it) \vert^p dm( \chi) =  \int_{ \mathbb{D}^{ \infty}} \vert D(f) ( T_t \chi ) \vert^p dm( \chi) \]
\[ =  \int_{ \mathbb{D}^{ \infty}} \vert D(f) ( \chi ) \vert^p dm( \chi) = \Vert D(f) \Vert_{ \mathcal{B}^p}^p. \]
We conclude using the Fubini's theorem.
\end{proof}

\begin{proposition}
Let $ \chi \in \mathbb{D}^{ \infty}$ and $f \in \mathcal{B}^p$. Then $g_{ \chi} \in H^p_i( \mathbb{C}_+)$ and $g_{ \chi}$ is an extension of $f_{ \chi}$ on $ \mathbb{C}_+$.
\end{proposition}

\begin{proof}
Thanks to the previous lemma, we already know that for almost every $ \chi \in \mathbb{D}^{ \infty}$, $g_{ \chi} \in L^p( \lambda_i)$. So it suffices to show (we use here a characterisation of the classical Hardy space) 
\[ \int_{- \infty}^{ \infty} \bigg{(} \frac{1-it}{1+it} \bigg{)}^n g_{ \chi}(it) d \lambda_{i}(t) = 0 \quad \hbox{ for } n \geq 1. \]

We use the same ideas than in \cite{gordon1999composition} and \cite{bayart2002hardy} but we have to adapt the proof because here we do not work with Fourier series on $ \mathbb{T}^{ \infty}$ but with functions in $B^p( \mathbb{D}^{ \infty})$. We fix $n \geq 1$ and define $G( \chi)$ by 
\[ G( \chi) := \int_{- \infty}^{ \infty} \bigg{(} \frac{1-it}{1+it} \bigg{)}^n g_{ \chi}(it) d \lambda_{i}(t). \]

Clearly $G \in L^p( \mathbb{D}^{ \infty})$ because 
$$\begin{array}{cl}
\dis\int_{ \mathbb{D}^{ \infty}} \vert G( \chi ) \vert^p dm( \chi) &\dis= \int_{ \mathbb{D}^{ \infty}} \bigg{ \vert } \int_{- \infty}^{ + \infty} \bigg{(} \frac{1-it}{1+it} \bigg{)}^n g_{ \chi}(it) d \lambda_i(t) \bigg{ \vert}^p dm( \chi)\\
&\dis\leq \int_{ \mathbb{D}^{ \infty}} \int_{- \infty}^{ + \infty} \vert g_{ \chi}(it) \vert^p d \lambda_i(t)  dm( \chi) = \Vert D(f) \Vert_{ \mathcal{B}^p}^p\\
\end{array}$$
where the last inequality follows from the preceding lemma. \\


Actually, $G \in B^p( \mathbb{D}^{ \infty})$. It suffices to show that there exists a sequence of analytic polynomials which converges to $G$. Since $f \in B^p( \mathbb{D}^{ \infty})$, we have $D(f) \in B^p( \mathbb{D}^{ \infty})$ and there exists  a sequence $(P_k)$ of analytic polynomials such that 
\[ \Vert D(f) - P_k \Vert_{ B^p( \mathbb{D}^{ \infty})} {\underset {k \rightarrow + \infty} { \longrightarrow}} 0. \]

Then we define the analytic polynomial
\[ Q_k( \chi) := \int_{- \infty}^{ + \infty} \bigg{(} \frac{1-it}{1+it} \bigg{)}^n P_k( T_t \chi ) d \lambda_i (t) \]
and we claim that $(Q_k)$ converges to $G$. Indeed
\[ \Vert G - Q_k \Vert^p_{ B^p( \mathbb{D}^{ \infty})} =  \int_{ \mathbb{D}^{ \infty}} \bigg{ \vert } \int_{- \infty}^{ + \infty} \bigg{(} \frac{1-it}{1+it} \bigg{)}^n ( g_{ \chi}(it)- P_{k, \chi}(it) )\,  d \lambda_i(t) \bigg{\vert}^p dm (\chi). \]
We get, through the Fubini's theorem, 
$$ \Vert G - Q_k \Vert^p_{ B^p( \mathbb{D}^{ \infty})} \leq \int_{- \infty}^{ + \infty} \Vert (D(f)-P_k)(T_t(\cdot)) \Vert^p_{ B^p( \mathbb{D}^{ \infty})} d \lambda_i(t)$$
but $T_t$ is just a rotation:
$$ \Vert G - Q_k \Vert^p_{ B^p( \mathbb{D}^{ \infty})} \leq  \int_{- \infty}^{ + \infty} \Vert (D(f)-P_k) \Vert^p_{ B^p( \mathbb{D}^{ \infty})} d \lambda_i(t) =  \Vert (D(f)-P_k) \Vert_{ B^p( \mathbb{D}^{ \infty})}^p$$
which goes to zero when $k$ goes to infinity, and this proves our claim.\\


We claim now that $G$ vanishes almost everywhere. Since $G$ belongs to $B^p( \mathbb{D}^{ \infty})$, it suffices to prove that $G$ is orthogonal to every  monomial with positive index. Let $q \in \mathbb{N}$. We have 
\[ \int_{\mathbb{D}^{ \infty}} \overline{\chi}(q) G( \chi) dm( \chi) = \int_{- \infty}^{ + \infty} 
\bigg{(} \frac{1-it}{1+it} \bigg{)}^n \int_{ \mathbb{D}^{ \infty}} \overline{\chi}(q) g_{ \chi}(it) dm ( \chi) \, d \lambda_i(t). \] 

Actually we have 
$$\int_{\mathbb{D}^{ \infty}} \overline{\chi}(q) g_{ \chi}(it) dm ( \chi) = 0$$
because $g_{\chi}(it) = Df(T_t \chi) \in B^p( \mathbb{D}^{ \infty})$. This is clear for the polynomials and by density this proves the claim. \\


The proof that $g_{ \chi}$ is an extension of $f_{ \chi}$ is the same than in the case of $ \mathcal{H}^p$ (see \cite{bayart2002hardy}).
\end{proof}

Now we shall denote $f_{ \chi}$ the extension instead of $g_{ \chi}$. Like in the case of $ \mathcal{H}^p$ with $p \geq 1$, this extension is almost surely simple. \\

\begin{proposition}{\label{fchiBp}}
Let $ \chi \in \mathbb{D}^{ \infty}$ and $f \in \mathcal{B}^p$ for $p \geq 1$. Then for almost every $ \chi$ (relatively to the measure $A$ on $ \mathbb{D}^{ \infty}$), $f_{ \chi}$ converges on $ \mathbb{C}_+$.
\end{proposition}

\begin{proof}
Let $f \in \mathcal{B}^2$ of the form $(1)$. We consider $ L^2(\mathbb{D}^{ \infty}, A)$ and the orthonormal sequence $ \Phi_n( \chi) = \sqrt{d(n)} \chi (n)$. For $ \sigma > 0$ and $t \in \mathbb{R}$, let $\dis c_n := \frac{a_n n^{- \sigma-it}}{\sqrt{d(n)}}\cdot$ We point out that $\dis\big(a_n/\sqrt{d(n)}\big)_{n\ge1}\in\ell^2$ and that $\dis\big(n^{-\sigma}\log(n)\big)_{n\ge1}\in\ell^\infty$ hence

$$\dis\sum_{n=1}^{ + \infty} \vert c_n \vert^2 \log^2(n) < + \infty .$$
So the Menchoff's lemma gives that $\sum c_n\Phi_n(\chi)$ converges for almost every $\chi$ . Therefore, we get the result when $p=2$. 

When $p \neq 2$, it suffices to prove the result for $p=1$. As in the case of the spaces $ \mathcal{A}^p$, the result follows from Prop.\ref{Tepsberg}.\end{proof}

Let $f \in \mathcal{B}^2$ , we know that for almost all $ \chi \in \mathbb{D}^{ \infty}$, $f_{ \chi}$ converges on $ \mathbb{C}_+$ and so $g_{\chi}=f_{ \chi}$, we obtain for each probability measure $w$ on $ \mathbb{R}$:
\[ \Vert f \Vert_{ \mathcal{B}^2}^2 = \int_{ \mathbb{R}} \int_{ \mathbb{D}^{ \infty}} \vert f_{ \chi} (it) \vert^2 dA( \chi) dw(t). \]

\begin{theoreme}{\label{Little2}}
Let $f \in \mathcal{B}^2$ and $w$ be a probability measure on $\mathbb{R}$. Then
\[ \Vert f \Vert_{ \mathcal{B}^2}^2 = \vert f( + \infty) \vert^2 + 4 \int_{ \mathbb{R}} \int_{0}^{ + \infty} \int_{ \mathbb{D}^{ \infty}} \sigma \vert f_{ \chi}( \sigma+it) \vert^2 dA( \chi) d \sigma dw(t). \]
\end{theoreme}

\begin{proof}
For $ \sigma >0 $, we have 
\[ \int_{\mathbb{D}^{ \infty}} \int_{ \mathbb{R}} \vert f_{ \chi}^{'}(\sigma+ it) \vert^2 dw(t) dm( \chi) =  \Vert f' \Vert_{ \mathcal{B}^2}^2 = \sum_{n=2}^{ + \infty} \frac{\vert a_n \vert^2 n^{-2 \sigma} \log^2(n)}{d(n)}. \]
We multiply by $ \sigma$ and it suffices to remark that:
\[\int_{0}^{ + \infty} \sigma n^{-2 \sigma} d \sigma = \frac{1}{4 \log^2(n)}\cdot\]
\end{proof}

\subsection{Inequalities on coefficients $ \mathcal{B}^p$}

We shall give here some inequalities between the $\mathcal{B}^p$ norm and some weighted $\ell^p$ norms of the coefficient of the functions (as in Th.\ref{theocoeffA}). 

\begin{theoreme}{\label{theocoeffB}}

Let $p \geq 1$. 
\begin{enumerate}[(i)]

\item When $1\le p\le2$ and $f=\dis\sum_{n\ge1} a_n \e_n\in \mathcal{B}^p$, we have 
$$\dis \Big\|\frac{a_n}{d(n)^{1/p}}\Big\|_{\ell^{p'}}\le\|f\|_{\mathcal{B}^p}. $$
\smallskip

\item When $p\ge2$ and $\dis\sum_{n\ge1}\frac{|a_n|^{p'}}{d(n)^{p'-1}}<\infty$, we have $f=\dis\sum_{n\ge1} a_n \e_n\in \mathcal{B}^p$ and
$$\dis \|f\|_{\mathcal{B}^p}\le\Bigg(\sum_{n\ge1} \frac{|a_n|^{p'}}{d(n)^{p'-1}}\Bigg)^{1/p'}=\Big\|\frac{a_n}{d(n)^{1/p}}\Big\|_{\ell^{p'}}. $$
\end{enumerate}
\end{theoreme}

\begin{proof} We do not give the details since it follows the same ideas as in the proof of Th.\ref{theocoeffA}. 

When $1\le p\le2$.

For every integer $n=p_1^{\alpha_1}p_2^{\alpha_2}\ldots\ge1$ and $f\in L^{p}({\mathbb D}^\infty,dA)$, let us define

$$\dis \tau_n(f)=\int_{{\mathbb D}^\infty}f(z) \bar z^{(n)}\,dA$$
where $\dis \bar z^{(n)}=\bar z_1^{\alpha_1}\bar z_2^{\alpha_2}\ldots$

We point out that, when $P$ is a Dirichlet polynomial $P(s)=\dis\sum_{n\ge1}a_n n^{-s}$, we associate as usual $f(z)=D(P)(z)=\dis\sum_{n\ge1}a_n z_1^{\alpha_1}z_2^{\alpha_2}\ldots$ We have in that case $\dis\tau_n(f)=a_n/d(n).$

Then we consider $\dis Q(f)=\big(\tau_n(f)\big)_{n\ge1}$. This defines norm one operators from $L^{1}({\mathbb D}^\infty,dA)$ to $L^\infty(\omega)$ and from $L^{2}({\mathbb D}^\infty,dA)$ to $L^2(\omega)$, with $\omega(n)=d(n)$. The same interpolation argument gives the conclusion.\end{proof}

\section{Annexe: around the norm of the point evaluation}

We wish to present here a principle to compare (relatively to $p$) the norm of the point evaluation. We shall work in a rather general framework of subspaces of functions of some $L^p$ spaces. When one work on classical spaces of analytic functions (Hardy-Bergman spaces), this principle is useless, since one can essentially work with any power of a function (up to some standard tools). In the context of Dirichlet series, a big difficulty is the fact that we have no way to consider $f^\alpha$ when $\alpha$ is not an integer (and $f\in{\cal D}$). The following method can be helpful and gives very precise result in some particular cases.

In this section, we consider some  subspaces $X_p\subset L^p(\Omega,\nu)$ of functions on $\Omega$, where $\nu$ is a probability measure on $\Omega$ and $p\ge1$. We assume that there exists some algebra ${\cal P}\subset\dis\cap_{p\ge1}X_p$ which is dense in each $X_p$ (think to the polynomials in many contexts).

We fix some $\omega\in\Omega$ and we assume that the point evaluation $f\in X_p\mapsto f(\omega)$ is bounded with norm $N_p$.

Let us mention that the most often, thanks to the theory of reproducing kernels, the value of $N_2$ is known (and easy to get).

We give here several very simple observations which we used in this paper.

\begin{proposition}
With the previous notations,

\begin{enumerate}[(i)]
\item If $\dis\frac{1}{p}=\frac{1}{q_1}+\frac{1}{q_2}$, then we have $\dis N_p\ge N_{q_1}N_{q_2}$.

\item Let $q\ge p\ge1$. We have $\dis N_p\ge N_{q}$.

\item Let $m$ be an integer. We have for every $p\ge1$, $\dis N_{pm}\le \big(N_{p}\big)^{1/m}$.

In particular, $\dis N_{2m}\le \big(N_{2}\big)^{1/m}$.
\end{enumerate}

\end{proposition}

\begin{proof}
$(i)$ Let $f$ and $g$ in ${\cal P}$ where $\dis\|f \|_{L^{q_1}}=1$ and  $\dis\|g \|_{L^{q_2}}=1$. The product $fg$ still belongs to ${\cal P}\subset X_p$ and we have 
$$\dis N_p\ge N_p\|fg\|_{L^p}= N_p\|fg\|_{X^p}\ge|f(\omega)|.|g(\omega)|.$$
Taking now the upper bound relatively to $f$ and to $g$, the first assertion follows.

$(ii)$ is trivial. 

$(iii)$ Let us point out that, by an obvious induction, we have ${\dis N_p\ge N_{q_1}\cdots N_{q_r}}$ as soon as $\dis\frac{1}{p}=\frac{1}{q_1}+\cdots+\frac{1}{q_r}\cdot$ In particular, we can write ${\dis\frac{1}{p}=\frac{1}{pm}+\cdots+\frac{1}{pm}}$ ($m$ times) so that $\dis N_{p}\ge\big( N_{pm}\big)^{m}.$\end{proof}

\bigskip

\goodbreak
\nocite{*}
\begin{footnotesize}
\bibliography{biblio2}
\bibliographystyle{plain}
\end{footnotesize}
{\small 
\noindent{\it 
Univ Lille-Nord-de-France UArtois, \\ 
Laboratoire de Math\'ematiques de Lens EA~2462, \\
F\'ed\'eration CNRS Nord-Pas-de-Calais FR~2956, \\
F-62\kern 1mm 300 LENS, FRANCE \\
maxime.bailleul@euler.univ-artois.fr\\
pascal.lefevre@univ-artois.fr}}

\end{document}